\documentclass[a4paper,11pt]{article}
\usepackage{graphicx} 
\usepackage{amsmath, relsize} 
\usepackage{todonotes} 
\usepackage{amsthm} 
\usepackage{amsfonts}
\usepackage{amssymb}
\usepackage{array}
\usepackage{url}
\usepackage{amsmath}
\usepackage{amssymb}
\usepackage{amsthm}
\usepackage{enumerate}
\usepackage{setspace}
\usepackage{graphicx,color,hyperref}
\usepackage[margin=1in]{geometry}
\parskip \medskipamount
\newtheorem{theorem}{Theorem}[section]

\newtheorem{lem}[theorem]{Lemma}
\theoremstyle{definition}
\newtheorem{definition}[theorem]{Definition}
\newtheorem{remark}[theorem]{Remark}

\newcommand{\ind}[1]{{\text{\Large $\mathfrak 1$}}\left(#1\right)}
\newcommand{\PR}{P}
\newcommand{\E}{E}
\newcommand{\Var}{\mathrm{Var}}
\newcommand{\lr}[1]{\left(#1\right)}
\newcommand{\lrb}[1]{\left[#1\right]}
\newcommand{\lrc}[1]{\left\{#1\right\}}

\newcommand{\compl}{\textrm{c}}

\begin{document}

\title{Critical density of activated random walks on transitive graphs}
\author{Alexandre Stauffer\thanks{University of Bath, Bath, UK; a.stauffer@bath.ac.uk. Supported  by a Marie Curie Career Integration Grant PCIG13-GA-2013-618588 DSRELIS, and an EPSRC Early Career Fellowship.} \and Lorenzo Taggi\thanks{Technische Universit\"at Darmstadt, Darmstadt, DE; taggi@mathematik.tu-darmstadt.de. Supported by a German Research Foundation Grant BE 5267/1.}}
\date{}
\maketitle

\begin{abstract}
   We consider the activated random walk model
   on general vertex-transitive graphs. A central question in this model is whether the critical density $\mu_c$ for sustained activity is strictly between 0 and 1. 
   It was known that $\mu_c>0$ on $\mathbb{Z}^d$, $d\geq 1$, and that $\mu_c<1$ on $\mathbb{Z}$ for small enough sleeping rate.
   We show that $\mu_c\to 0$ as $\lambda\to 0$ in all vertex-transitive transient graphs, implying that $\mu_c<1$ for small enough sleeping rate.
   We also show that $\mu_c<1$ for any sleeping rate 
   in any vertex-transitive graph in which simple random walk has positive speed. Furthermore, we prove that $\mu_c>0$ in any vertex-transitive amenable graph, and that $\mu_c\in(0,1)$ for any sleeping rate on regular trees.
\end{abstract}

\section{Introduction}\label{sec:intro}
We consider the activated random walk (ARW) model on 
a graph $G=(V,E)$. 
This is a continuous-time interacting particle system with conserved number of particles, 
where each particle can be in one of two states: A (active) or S (inactive, sleeping). 
Initially, the number of particles at each vertex of $G$ is an independent Bernoulli random variable of parameter $\mu\in(0,1]$, 
usually called the \emph{particle density}, and all particles are
of type A. 
Each A-particle performs an independent, continuous time
random walk on $G$ with jump rate $1$, and
with each jump being to a uniformly random neighbor.
Moreover, every A-particle has a Poisson clock of rate $\lambda>0$
(called the \emph{sleeping rate}).
When the clock of a particle rings,
if the particle does not share
the site with other particles,
the transition $A \rightarrow S$ occurs (that is, the particle becomes of type S); otherwise nothing happens.
Each S-particle does not move and remains sleeping until another particle jumps into its location.
At such an instant, the S-particle turns into type A, giving the transition A+S $\rightarrow$ 2A.

For any given $\lambda$, it is expected that ARW undergoes a phase transition as $\mu$ varies. For example,
if $\mu$ is very small, there is a lot of empty space between particles, which allows each particle to eventually 
fall asleep (that is, turn into type $S$) and never become active again. When this happens, we say that \emph{ARW fixates}. 
When this does not happen, we say that \emph{ARW is active}. This case is expected to occur when $\mu$ is large, since active particles will 
repetitively jump on top of other particles, ``waking up'' the ones that had turned into type $S$. 

In a seminal paper, Rolla and Sidoravicius~\cite{Rolla} showed that this process satisfies a 0-1 law (i.e., the process is either 
active or fixated with probability 1) and is monotone with respect to $\mu$. This gives the existence of a critical value 
\begin{equation}
   \mu_c=\mu_c\left(\lambda\right)  := \inf\lrc{ \mu \geq 0 \, : \, \PR(\text{ARW is active})  > 0   }
   \label{eq:criticaldensity}
\end{equation}
such that ARW is active almost surely for all $\mu > \mu_c$, and fixates almost surely for all $\mu<\mu_c$. 
Though~\cite{Rolla}, as well as almost all existing works, are restricted to the case of $G$ being $\mathbb{Z}^d$, 
the above properties hold for any vertex-transitive graph. 
A graph $G$ is called \emph{vertex transitive} if, for any two vertices $u,v\in V$, there exists a graph automorphism of $G$ mapping $u$ onto $v$.
Throughout this paper we always consider that $G$ is an infinite graph that is locally finite and \emph{vertex transitive}, which ensures the existence of $\mu_c$.


Our definition above implies that $\mu_c\leq 1$ since particles are initially distributed as Bernoulli random variables. However, even if we replace this with any product measure of density $\mu>0$, 
it is intuitive that $\mu_c\leq 1$, since at most one particle can fall asleep at any given vertex.
This has been established for a large class of graphs~\cite{Rolla,Shellef,Amir}.
A fundamental and very important problem in activated random walks~\cite{Rolla,Dickman} is whether
\begin{equation}
   \text{$\mu_c\in(0,1)$ for all $\lambda> 0$ and all vertex-transitive graphs.}
   \label{eq:question}
\end{equation}
This problem is widely open, 
and both sides of the above question (i.e., whether $\mu_c>0$ or $\mu_c<1$) turned out to be 
very challenging. 
In fact, even showing that $\mu_c<1$ for \emph{some} $\lambda>0$ is already quite difficult, and is 
only known to hold on $\mathbb{Z}$, thanks to a very recent paper by Basu, Ganguly and Hoffman~\cite{Basu}.
Our first theorem establishes this result in all vertex-transitive \emph{transient} graphs, 
which includes $\mathbb{Z}^d$, $d\geq 3$.
\begin{theorem}
   \label{theo:transient_graphs}
   For any vertex-transitive, transient graph, it holds that 
   $$
      \mu_c \to 0 \text{ as } \lambda\to0.
   $$
   More specifically, we have that $\lim_{\lambda\downarrow 0}\frac{\mu_c}{\lambda^{1/4}}<\infty$.
\end{theorem}

Regarding the question of whether $\mu_c<1$ for \emph{all} $\lambda>0$,
this has until now not been established for any single graph.
A positive answer to this question
has only been given for a variant of ARW where particles move according to \emph{biased} 
random walks on $\mathbb{Z}^d$; see Taggi~\cite{Taggi}. 
Rolla and Tournier~\cite{Rolla2} further extended
this result by proving that, for biased ARW on $\mathbb{Z}^d$, we have $\mu_c(\lambda)\to0$ as $\lambda\to 0$. 
In our second theorem we give a positive answer 
to this question for the original, unbiased model, and for all 
graphs where simple random walk has positive speed.
If $(X_t)_{t\in\mathbb{N}}$ is a random walk on $G$ starting from a vertex $x$, and $|X_t|$ denotes the distance between $X_t$ and $x$,
we say that a random walk on $G$ has \emph{positive speed} if $\liminf_{t\to\infty} \frac{|X_t|}{t}>0$ almost surely. 
This includes, for example, all non-amenable graphs that are vertex transitive.
\begin{theorem}
   \label{theo:activity_non_amenable}
   For any vertex-transitive graph such that a random walk on it has positive speed $\alpha$, it holds that
   $$
      \mu_c <1 \quad\text{ for all } \lambda>0.
   $$
   More specifically, we obtain that $\mu_c < 1 - \frac{ \alpha\delta }{1 + \lambda}$,
   where $\delta$ is the probability that a random walk does not return to the origin. 
\end{theorem}
We prove the theorem above by providing general sufficient conditions for ARW to be active, which as a consequence establishes an upper bound on $\mu_c$. We believe this result is of independent interest 
and state it in Theorem~\ref{theo:suffcondnonamenable}. 

For the other side of~\eqref{eq:question} (i.e., whether $\mu_c>0$), 
there has been a bit more progress.
It has been settled when $G$ is $\mathbb{Z}^d$ thanks to the seminal work of 
Rolla and Sidoravicius~\cite{Rolla} for $d=1$, and an elaborate proof of 
Sidoravicius and Teixeira~\cite{Sidoravicius} for $d\geq 2$.
Our next theorem establishes that $\mu_c>0$ in any vertex-transitive \emph{amenable} graph, which includes $\mathbb{Z}^d$, $d\geq1$.
We remark that not only our result generalizes the ones in~\cite{Rolla,Sidoravicius}, but also provides the additional information
that $\mu_c \rightarrow 1$ as $\lambda \rightarrow \infty$. 
In addition, our proof is quite short in comparison to~\cite{Sidoravicius}.
\begin{theorem}
   \label{theo:mu_c_positive_amenable}
   For any vertex-transitive, amenable graph, we have 
   $$
      \mu_c >0 \quad\text{ for all } \lambda>0.
   $$
   More specifically, we have $\mu_c \geq \frac{\lambda}{1+\lambda}$.
\end{theorem}
\begin{remark}
   Our lower bound is sharp, in the sense that there are no better
   lower bounds for $\mu_c$ which are just a function
   of $\lambda$ and hold for any vertex-transitive amenable
   graph and any jump distribution.
   Indeed, $\mu_c$
   is known to be equal to $ \frac{\lambda}{1+\lambda}$
   on $\mathbb{Z}$ with totally asymmetric jumps
   \cite{Hoffman}.
\end{remark}

\begin{remark}
   Our Theorems~\ref{theo:transient_graphs} and~\ref{theo:mu_c_positive_amenable} hold in more generality,  
   for any distribution of the initial location of the particles
   and for any jump distribution (biased or unbiased) which is translation invariant
   and has finite support.
\end{remark}

Note that Theorems~\ref{theo:activity_non_amenable} and~\ref{theo:mu_c_positive_amenable} provide 
a final answer to~\eqref{eq:question}
in vertex-transitive graphs that are amenable but for which a random walk has positive speed; for example, the so-called lamplighter graphs.
In our final result, we also establish~\eqref{eq:question} for the case of regular trees, excluding $\mathbb{Z}$. 
\begin{theorem}
   \label{theo:tree}
   When $G$ is a regular tree of degree at least 3, we have
   $$
      \mu_c\in(0,1) \quad\text{ for all }\lambda>0.
   $$
   In addition, we have $\mu_c\to 0$ as $\lambda\to 0$.
 \end{theorem}

We now give a brief description of our proof techniques. The traditional strategy to establish bounds on $\mu_c$ is to consider a 
ball $B_L\subset V$ of some large radius $L$, centered at a given vertex $x\in V$, and \emph{stabilize} ARW inside this ball. This consists of
letting the process run (i.e., particles move and fall asleep) inside $B_L$, 
deleting every particle that exits $B_L$. This procedure will eventually end. At this point, each vertex of $B_L$ will either contain a 
sleeping particle or contain no particle; such a vertex is usually called \emph{stable}. 
It was shown in~\cite{Rolla}
that, roughly speaking, ARW is active if and only if the number of times particles visit $x$ during the stabilization of $B_L$ 
goes to infinity with $L$. In this paper, we introduce a 
new point of view on such stabilization procedure by focusing on some vertex $y \in B_L$, and carrying out what we call a \emph{weak stabilization of $B_L$
with respect to $y$}. Intuitively, in the weak stabilization we perform the steps of a stabilization procedure until each vertex of $B_L\setminus \{y\}$ is stable 
while $y$ is allowed to be either stable or host exactly one active particle. 
This strategy allows us to estimate the probability that, at the end of a 
stabilization procedure, $y$ contains a sleeping particle. In principle, the density of sleeping particles should correspond to $\mu_c$, and it is 
by controling such probability that we obtain estimates on $\mu_c$.
We believe that our weak stabilization procedure and our point of view of estimating the density of sleeping particles have the potential to foster 
even more substantial progress in this model. In fact, we believe our estimate on the probability that a sleeping particle ends at some vertex is 
of independent interest, and we state it in Theorem~\ref{theo:boundsQ}. 

The remaining of the paper is organized as follows. 
In Section~\ref{sec:Diaconis}, we describe the so-called Diaconis-Fulton representation of ARW and its properties, which we employ in all of our proofs.
Then, in Section~\ref{sec:weak_stabilization}, 
we introduce the weak stabilization procedure and estimate the probability of having a sleeping particle at a given vertex (Theorem~\ref{theo:boundsQ}).
Next we turn to the proofs of our main results: we 
prove 
Theorem~\ref{theo:transient_graphs} in Section~\ref{sec:proof_of_theorem_transient_graphs},
Theorem~\ref{theo:activity_non_amenable} in Section~\ref{sec:proof_of_theorem_activity_non_amenable},
Theorem \ref{theo:mu_c_positive_amenable} in Section \ref{sec:proof_of_theorem_mu_c_positive_amenable}, 
and Theorem~\ref{theo:tree} in Section~\ref{sec:proof_of_theo_tree}.

\section{Diaconis-Fulton representation}
\label{sec:Diaconis}
In this section we describe the
Diaconis-Fulton graphical representation
for the dynamics of ARW, following~\cite{Rolla}. 
For a graph $G=(V,E)$, the state of configurations is $\Omega=\{0,\rho,1,2,3,\ldots\}^V$, where a vertex being in state $\rho$ denotes that the vertex has 
one sleeping particle, while being in state $i\in\{0,1,2,\ldots\}$ denotes that the vertex contains $i$ active particles.
We employ the following order on the states of a vertex: $0 < \rho < 1<2<\cdots$.
In a configuration $\eta\in \Omega$,
a site $x \in V$ is called \textit{stable} if
$\eta(x) \in \{0, \rho \}$,
and it is called \textit{unstable} if $\eta(x) \geq 1$.
We fix an array of  \textit{instructions} 
$\tau = ( \tau^{x,j}: \, x \in V, \, j \in \mathbb{N})$,
where $\tau^{x,j}$ can either be of the form $\tau_{xy}$
or $\tau_{x\rho}$. We let $\tau_{xy}$ with $x,y\in V$ denote the instruction that a particle from $x$ jumps to vertex $y$, and $\tau_{x\rho}$ denote
the instruction that a particle from $x$ falls asleep.
Henceforth we call $\tau_{xy}$ a \emph{jump instruction} and $\tau_{x\rho}$ a \emph{sleep instruction}.
Therefore, given any configuration $\eta$, performing the instruction $\tau_{xy}$ in $\eta$ yields another configuration $\eta'$ such that 
$\eta'(z)=\eta(z)$ for all $z\in V\setminus\{x,y\}$, $\eta'(x)=\eta(x)-\ind{\eta(x)\geq 1}$, and $\eta'(y)=\eta(y)+\ind{\eta(x)\geq 1}$. We use the convention that $1+\rho=2$.
Similarly, performing the instruction $\tau_{x\rho}$ to $\eta$ yields a configuration $\eta'$ such that 
$\eta'(z)=\eta(z)$ for all $z\in V\setminus\{x\}$, and if $\eta(x)=1$ we have $\eta'(x)=\rho$, otherwise $\eta'(x)=\eta(x)$.

Let $h = ( h(x)\, : \,  x \in V)$ count the number of 
instructions used at each site.
We say that we \textit{use} an instruction
at $x$ (or that we \emph{topple} $x$) when we act on the current
particle configuration $\eta$ through the operator $\Phi_x$,
which is defined as,
\begin{equation}
\label{eq:Phioperator}
\Phi_x ( \eta, h) =
( \tau^{x, h(x) + 1}  \, \eta, \, h + q_x).
\end{equation}
The operation $\Phi_x$ is \textit{legal} for $\eta$ if $x$ is unstable in $\eta$, in which case we set $q_x=1$, 
otherwise it is \textit{illegal} and we set $q_x=0$.

\vspace{0.8cm}

\noindent \textit{\textbf{Properties.}}
We now describe the properties of this representation.
Later  we discuss how they are related to the the stochastic dynamics of ARW.
For a sequence of vertices $\alpha = ( x_1, x_2, \ldots x_k)$,
we write $\Phi_{\alpha} = \Phi_{x_k} \Phi_{x_{k-1}}
\ldots \Phi_{x_1}$ and we say that $\Phi_{\alpha}$ is
\textit{legal} for $\eta$ if $\Phi_{x_\ell}$
is legal for $\Phi_{(x_{\ell-1}, \ldots, x_1)} (\eta,h) $
for all $\ell \in \{ 1, 2, \ldots k \}$.
Let $m_{\alpha} = ( m_{\alpha}(x) \, : \,x \in  V )$
be given by,
 $m_{\alpha}(x) \, = \, \sum_{\ell} \ind{x_\ell = x},$
the number of times the site $x$ appears in $\alpha$.
We write $m_{\alpha} \geq m_{\beta}$ if
$m_{\alpha} (x)  \,  \geq \, m_{\beta} (x) \, \, \, \forall x \in V$.
Analogously we write $\eta'   \geq   \eta$ if $\eta' (x) \, \geq \, \eta(x)$
for all $x \in V$. We also write $(\eta', h') \geq (\eta, h)$
if $\eta' \geq \eta$ and $h' = h$.

Let $\eta, \eta'$ be two configurations, $x$ be a vertex in $V$
and  $\tau$ be an array of instructions. 
Let $V'$ be a finite subset of $V$. A configuration $\eta$ is said to be \textit{stable} in $V'$
if all the sites $x \in V'$ are stable. We say that $\alpha$ is contained in $V'$
if all its elements are in $V'$, and we say that $\alpha$ \textit{stabilizes} $\eta$ in $V'$
if every $x \in V'$ is stable in $\Phi_\alpha \eta$.
The following lemmas give fundamental properties of the Diaconis-Fulton representation. 
For the proof we refer to \cite{Rolla}.

\begin{lem}[Least Action Principle]\label{prop:lemma1}
   If $\alpha$ and $\beta$ are legal sequences of topplings for $\eta$ such that $\beta$ is contained in $V$
   and $\alpha$ stabilizes $\eta$ in $V$, then $m_{\beta} \leq m_{\alpha}$.
\end{lem}

\begin{lem}[Abelian Property]\label{prop:lemma2}
   Given any $V'\subset V$,
   if $\alpha$ and $\beta$ are both legal sequences for $\eta$
   that are contained in $V'$ and stabilize $\eta$ in $V'$, 
   then $m_{\alpha} = m_{\beta}$. In particular, $\Phi_{\alpha} \eta = \Phi_{\beta} \eta$.
\end{lem}

For any finite subset $V' \subset V$, any $x\in V'$, any particle configuration $\eta$, and any array of instructions $\tau$, we denote by $m_{V^{\prime},\eta,\tau}(x)$ the number of times that $x$ is toppled in the stabilization of $V'$ starting from configuration $\eta$ and using the instructions in $\tau$. Note that by Lemma~\ref{prop:lemma2}, we have that $m_{V^{\prime},\eta,\tau}$ is well defined.
\begin{lem}[Monotonicity]\label{prop:lemma3}
   If $V' \subset V''\subset V$ and $\eta \leq \eta'$, then $m_{V', \eta, \tau} \leq m_{V'', \eta', \tau}$.
\end{lem}

By monotonicity, given any growing sequence of subsets $V_1\subseteq V_2 \subseteq V_3\subseteq \cdots \subseteq V$ such that $\lim_{m\to\infty} V_m=V$, 
the limit 
$$
   m_{\eta, \tau} = \lim\limits_{m\to \infty} m_{V_m, \eta, \tau},
$$ 
exists and does not depend
on the particular sequence $\{V_m\}_m$.

We now introduce a probability measure on the space of instructions and of particle configurations.
We denote by $\mathcal{P}$ the probability measure according to which,
for any $x \in V$ and any $j \in \mathbb{N}$,
$\mathcal{P} (  \tau^{x,j} = \tau_{x\rho}   ) = \frac{\lambda}{1 + \lambda}$ and 
$\mathcal{P} (  \tau^{x,j} = \tau_{xy}   ) = \frac{1}{d(1 + \lambda)}$ for any $y\in V$ neighboring $x$,
where $d$ is the degree of each vertex of $G$ and the $\tau^{x,j}$ are independent across diffent values of $x$ or $j$.
Finally we denote by $P^\nu=\mathcal{P}\otimes \nu$ the joint law of
$\eta$ and $\tau$, where $\nu$ is a distribution on $\Omega$ giving the law of $\eta$.
Let $\mathbb{P}^\nu$ denotes the probability measure induced by the ARW process when the initial distribution of particles is given by $\nu$. 
We shall often omit the dependence on $\nu$ by writing $P$ and $\mathbb{P}$ instead of $P^\nu$ and $\mathbb{P}^\nu$.
The following lemma relates the dynamics of ARW to the stability property of the representation.
\begin{lem}[0-1 law]
   \label{prop:lemma4}
   Let $\nu$ be an automorphism invariant, ergodic distribution with finite density.
   Let $x\in V$ be any given vertex of $G$.
   Then $\mathbb{P}^{\nu}  (\text{ARW fixates} ) = P^{\nu} ( m_{\eta, \tau} (x) < \infty ) \in \{0, 1 \}$.
\end{lem}

Roughly speaking, the next lemma gives that removing a sleep instruction, cannot decrease the number of instructions 
used at a given vertex for stabilization.
In order to state the lemma, consider an additional instruction $\iota$ besides $\tau_{xy}$ and $\tau_{x\rho}$. The effect of $\iota$ is to 
leave the configuration unchanged; i.e., $\iota \, \eta = \eta$.  
Then given two arrays $\tau = \left( \tau^{x,j} \right)_{x ,\, j }$ 
and $\tilde{\tau} = \left( \tilde{\tau}^{x,j} \right)_{x, \, j }$,
we write $\tau \leq \tilde{\tau}$ if for every $x \in V$ and $j \in \mathbb{N}$,
we either have $\tilde{\tau}^{x,j} = {\tau}^{x,j}$ or we have $\tilde{\tau}^{x,j} = \iota$ and 
${\tau}^{x,j} =  \tau_{x\rho}$.

\begin{lem}[Monotonicity with enforced activation]
   \label{prop:lemma5}
   Let $\tau$ and $\tilde{\tau}$ be two arrays of instructions such that $\tau \leq \tilde{\tau}$.
   Then, for any finite subset $V' \subset V$ and configuration $\eta \in \Omega$,  we have
   $m_{V', \eta, \tau} \leq m_{V', \eta, \tilde{\tau}}.$
\end{lem}

When we average over $\eta$ and $\tau$ using the measure $P$, we will simply write $m_{V'}$ instead of $m_{V',\eta,\tau}$.

\section{Weak stabilization}\label{sec:weak_stabilization}
In this section we introduce our method of weak stabilization and use it to derive upper and lower bounds on the probability that a given vertex contains an $S$-particle at the end of the stabilization of some set. This is the content of Theorem \ref{theo:boundsQ} below, which will play a fundamental role in the proofs of our main results.
For any finite set $ K \subset V$ and any vertex $x \in K$, let $Q(x, K)$ be the probability that there is one $S$-particle at $x$ at the end of
the stabilization of $K$.
\begin{theorem}
\label{theo:boundsQ}
Consider ARW on a vertex-transitive graph $G=(V, E)$.
Then, for any $K\subset V$ and any $x\in K$, we have
\begin{equation}
   Q(x,K) \geq \frac{\lambda}{1+\lambda} P( m_{K}(x) \geq 1).
   \label{eq:lbound}
\end{equation}
Moreover, if $G$ is a vertex-transitive, transient graph, then 
\begin{equation}
   Q(x,K)\leq 
      3\sqrt{\lambda\lr{C_G(1+\lambda)+1}},
   \label{eq:ubound}
\end{equation}
where $C_G$ is the expected number of times a simple random walk on $G$ starting from $x$ visits $x$.
\end{theorem}

In the proof of the theorem above we will employ the notion of \emph{weakly stable configurations} and \emph{weak stabilization}.

\begin{definition}[weakly stable configurations]\label{def:wstable}
	We  say that a configuration $\eta$ is \emph{weakly stable} in a subset $K\subset V$ with respect to a vertex $x\in K$ 
	if $\eta(x)\leq 1$ and $\eta(y)\leq\rho$ for all $y\in K\setminus\{x\}$. 
	In words, this means that all vertices in $K\setminus \{x\}$ are stable, and $x$ is either stable or hosts at most one active particle.
	For conciseness, we just write that $\eta$ is weakly stable for $(x, K)$.
\end{definition}

\begin{definition}[weak stabilization]
   Given a subset $K\subset V$ and a vertex $x\in K$, the \emph{weak stabilization} of $(x,K)$ is a sequence of topplings of unstable sites of $K\setminus\{x\}$ and of topplings of $x$ whenever $x$ has at least two active particles, until a weakly stable configuration for $(x,K)$ is obtained. The order of the topplings of a weak stabilization can be arbitrary.
\end{definition}

The notion of {legal instructions} can be extended to weak stabilization of $(x,K)$ as follows.
We call a vertex $y$ unstable in $\eta \in \Omega$ if $\eta(y) \geq 1 + \delta_x(y)$,
where $\delta_x(y)=1$ if $x=y$ and $\delta_x(y)=0$ otherwise.
We call the operation $\Phi_y$ defined in (\ref{eq:Phioperator})
\textit{legal} for $\eta$ if $y$ is unstable in $\eta$ \textit{illegal}
otherwise.

\begin{remark}
\label{remark:weakstab}
Consider any finite subset $K \subset V$ and any $x \in K$. The Least Action Principle (Lemma \ref{prop:lemma1}), the Abelian property (Lemma \ref{prop:lemma2}), Monotonicity (Lemma \ref{prop:lemma3}), and Monotonicity with enforced activation (Lemma \ref{prop:lemma5}) hold true for weak stabilization of $(x,K)$ as well with no change in the proof.
\end{remark}



The main idea of the proof of Theorem~\ref{theo:boundsQ} is to 
perform a certain sequence of topplings to stabilize $K$ that will allow us to control whether there is a sleeping particle at $x$.
From the Abelian property (Lemma~\ref{prop:lemma2}), in order to stabilize $K$ we can perform the topplings in any order we want. 
We will stabilize $K$ by first weakly stabilizing $(x,K)$, which gives a weakly stable configuration $\eta_1$ for $(x,K)$. Then either $\eta_1$ is stable for $K$, in which case we finish the stabilization procedure, or $\eta_1(x)=1$. In the latter case, we topple $x$ and weakly stabilize $(x,K)$ again, obtaining a configuration $\eta_2$. We repeat the above procedure until we obtain a stable configuration for $K$, concluding the stabilization. We will refer to this stabilization procedure as a \emph{stabilization via weak stabilization}.

\subsection{Proof of the lower bound in Theorem~\ref{theo:boundsQ}}
Note that, in a stabilization via weak stabilization, after each weakly stable configuration $\eta_i$ we obtain, 
if $\eta_i$ is not stable, then with probability $\frac{\lambda}{1 + \lambda}$ we encounter a sleep instruction at $x$, transforming $\eta_i$ into a stable configuration.
With this we can derive the lower bound~\eqref{eq:lbound} in Theorem~\ref{theo:boundsQ}.

\begin{proof}[Proof of~\eqref{eq:lbound} in Theorem~\ref{theo:boundsQ}]
   We apply the stabilization of $K$ via weak stabilizations of $(x,K)$. Let $\eta_1$ be the first weakly stable configuration for $(x,K)$ that is obtained in this procedure.
	As discussed above, if $\eta_1$ is not stable for $K$, then we obtain a stable configuration for $K$ if the next instruction at $x$ is sleep. Hence, 
	$$
	   Q(x,K) \geq P(\eta_1 \text{ is not stable for $K$}) \frac{\lambda}{1+\lambda}.
	$$
	The proof is concluded by noting that the event that $\eta_1$ is not stable for $K$ is equivalent to the event that $x$ is toppled at least once. 
	This is true because of the following. 
	If $\eta_1$ is not stable for $K$, then $\eta_1(x)=1$ which implies that $x$ will be toppled at least once. 
	In the other direction, if $x$ is toppled at least once, then 
	this happens either before $\eta_1$ is obtained or because $\eta_1(x)=1$. 
	But if $x$ was toppled before $\eta_1$ was obtained, this must have happened at a time when $x$ had at least
	two particles. From this time onwards, $x$ will have at least one active particle until $\eta_1$ is obtained. Hence, $\eta_1$ is not stable.
\end{proof}

\subsection{Proof of the upper bound in Theorem~\ref{theo:boundsQ}}
Our proof of the upper bound~\eqref{eq:ubound} for $Q(x,K)$ is a bit longer than the proof of the lower bound. 
We will perform the stabilization of $K$ via weak stabilization as described above.
The idea is to estimate the probability that, for any $i\geq1$, we obtain a stable configuration for $K$ after the $i$th weak stabilization of $(x,K)$. We do this by relating this probability to the probability that a random walk starting from $x$ never returns to $x$. It is at this step that we use that $G$ is transient.

After the $i$th time we perform the weak stabilization of $(x,K)$, we let $m^{i}_{(x,K)}(y)$ be the number of instructions
that have been used at $y \in K$ up to this time, and denote by $\eta_i$ the configuration we then obtained. 
Also, let $T_{(x,K)}$ denote the number of weak stabilizations of $(x,K)$ we perform until a stable configuration in $K$ is obtained.
Note that $\eta_{T_{(x,K)}}$ is either a stable configuration, which implies that $\eta_{T_{(x,K)}}(x)=0$, or $\eta_{T_{(x,K)}}$ is weakly stable for $(x,K)$ with $\eta_{T_{(x,K)}}(x)=1$ and the next 
instruction used at $x$ was a sleep instruction, thereby concluding the stabilization of $K$.
For consistency, for any $i > T_{(x, K)}$, let $\eta_i$ be the stable configuration obtained after stabilizing $K$ and, for any $y \in K$, define  $m^{i}_{(x, K)}(y)=m_K(y)$, which is 
the total number of instructions used
at $y$ for the complete stabilization of $K$.
By the Abelian property, the quantities $T_{(x,K)}$ and $m^i_{(x,K)}$ are all well defined.

Below we state a lemma, and then show how this lemma implies the upper bound on $Q(x,K)$.
\begin{lem}\label{lem:expt}
   Given any vertex-transitive, transient graph $G=(V,E)$, any subset $K\subset V$ and any vertex $x\in K$, 
   and letting $C_G$ be the expected number of visits to $x$ of a random walk on $G$ starting from $x$, we have
$$
   E[T_{(x,K)}] \leq C_G (1 + \lambda)+1,
$$
where the expectation is with respect to the measure $P$.
\end{lem}

\begin{proof}[Proof of~\eqref{eq:ubound} in Theorem~\ref{theo:boundsQ}]
   For simplicity, write $\eta'=\eta_{T_{(x,K)}+1}$ for the configuration obtained after complete stabilization of $K$.
   Then the following expression holds, as the sum is over disjoint events,
	\begin{align}
		\label{eq:sumoverk}
		Q(x,K) = P(\eta'(x)=\rho)=
		 \sum\limits_{k=1}^{\infty} 
		P \left(  T_{(x,K)} = k, \, \, \eta'(x) = \rho \right).
	\end{align}
	 Now observe that
	\begin{equation}
		\label{eq:inter}
		 P \left(   T_{(x,K)} = k,  \, \, \,   \eta'(x) = \rho   \right) 
		\leq \left( \frac{1}{1+\lambda} \right)^{k-1} \frac{\lambda}{1+ \lambda}.
	\end{equation}
	The previous inequality follows from independence of instructions: the event in the left-hand side implies that 
	after each weak stabilization we have an active particle at $x$, and moreover we encounter a jump instruction at $x$ after each of the first $k-1$ weak stabilizations, and a 
	sleep instruction at $x$ after the last weak stabilization.
	Hence, for any $H\geq 1$ we can write
	\begin{align*}
		Q(x,K)
	    & \leq   \frac{\lambda}{1+\lambda}\sum\limits_{k=1}^{H} {\left(  \frac{1}{1 +\lambda  } \right)}^{k-1}   \, + \,  P( T_{(x,K)} > H) \\
	    & \leq 1 - \left( \frac{1}{1+\lambda}\right) ^{H} + \frac{ E[T_{(x,K)}]}{H},
	\end{align*}
	where in the last step we used Markov's inequality.
	From Lemma~\ref{lem:expt}, we obtain 
	\begin{align*}
	   Q(x,K) 
	   & \leq  1 - \left( \frac{1}{1+\lambda}\right) ^{H} + \frac{C_G(1 + \lambda)+1}{H} \\
      & \leq 1 - \lr{1-\lambda}^{H} + \frac{C_G(1 + \lambda)+1}{H} \\
      & \leq \lambda H + \frac{C_G (1+ \lambda) +1}{H}.
	\end{align*}
Observe that our estimate holds for any integer $H \geq 1$ and that $C_G \geq 1$. Then, by setting $H= \lfloor \sqrt{\frac{C_G(1+\lambda)+1}{\lambda}}  \rfloor$,  we get that for any positive $\lambda$,
	$$
	   Q(x,K) 
	   \leq  	   3 \sqrt{\lambda\lr{C_G(1+\lambda)+1}}.
	$$
	In the above calculations we used  $\frac{x}{ \lfloor x \rfloor} \leq 2$ for $x \geq 1$. 
\end{proof}

\subsection{Proof of Lemma~\ref{lem:expt}}
In this section we establish the upper bound on $E[T_{(x,K)}]$ from Lemma~\ref{lem:expt}.
Let $x\in V$ be a given vertex. 
Let $E^x$ denote the expectation $E$ conditioned on the initial particle configuration having one active particle at $x$, and 
$E_x$ denote the expectation conditioned on the initial particle configuration having no particle at $x$. 
\begin{lem}\label{lem:m}
   For any finite subset $K\subset V$ and vertex $x\in K$ we have
	$$
	    E_x[m_{K}(x)] 
		\leq  E^x[m_{(x,K)}^1(x)].
	$$
\end{lem}
\begin{proof}
   Consider an initial particle configuration $\eta$ having no particle at $x$, and the particle configuration $\eta^x$ obtained from $\eta$ by adding 
   an active particle at $x$. We will show a stronger result saying that, by using the same instruction array for both $\eta$ and $\eta^x$, $m_K(x)$ starting 
   from $\eta$ is at most $m^1_{(x,K)}(x)$ starting from $\eta^x$.
   We stabilize $K$ starting from $\eta$ via weak stabilization of $(x,K)$, and do the same topplings for $\eta^x$. 
   Since $\eta$ and $\eta^x$ differ only at $x$, until the first weak stabilization of 
   $\eta$ is concluded, the same topplings can be carried out in $\eta^x$ as well. 
   At this point, 
   if there is a particle at $x$ in $\eta$, there are two particles at $x$ in $\eta^x$. 
   Then if the next instruction at $x$ is a jump instruction, we can perform the same toppling in $\eta$ and $\eta^x$, and we repeat this procedure until
   another weakly stable configuration is obtained in $\eta$.
   On the other hand, if the next instruction at $x$ is a sleep instruction, then the stabilization of $\eta$ is concluded, but the 
   weak stabilization of $\eta^x$ continues. 
   Finally, if there is no particle at $x$ at the end of a weak stabilization of $\eta$, then the stabilization of $\eta$ and the weak stabilization of 
   $\eta^x$ are concluded.
   Therefore, under this coupling, the weak stabilization of $\eta$ concludes no later than that of $\eta^x$, concluding the proof.
\end{proof}

\begin{proof}[Proof of Lemma~\ref{lem:expt}]
	The crucial observation is the following. 
	Assume that $T_{(x,K)}\geq 2$. After each of the first $T_{(x,K)}-1$ weak stabilizations of $(x,K)$, we must perform at least one toppling at $x$, 	and this toppling happens after the first weak stabilization of $(x,K)$, so it is not counted in $m^{1}_{(x,K)}(x)$. This gives that 
	\begin{equation}
		\label{eq:bound1}
		T_{(x,K)} -1 \leq m_{K}(x) - m^{1}_{(x,K)}(x).
	\end{equation}
	The above bound also holds when $T_{(x,K)}=1$ since $m_{K}(x)\geq m^{1}_{(x,K)}(x)$.
	Then the lemma follows by claiming that 
	\begin{align}
		\label{eq:m}
        E[m_{K}(x)] 
		\leq  E[m_{(x,K)}^1(x)] + C_G ( 1 + \lambda).
	\end{align}
	First we prove~\eqref{eq:m} with $E$ replaced with $E^x$. 
	Denote the particle that starts at $x$ by $z$. 
    From Lemma~\ref{prop:lemma5}, we have that if we ignore some sleep instructions during the stabilization of $K$ (i.e., we replace some sleep instructions in the    instruction array $\tau$ with neutral instructions $\iota$), 
    the value of $m_K(x)$ can only increase.
    Therefore, we can bound $m_K(x)$ from above by carrying out a two-step stabilization procedure. 
    In the first step, we move $z$ ignoring any sleep instruction seen until $z$ exits $K$. 
    We call $\mathcal{V}$ the expected number of topplings at $x$ up to this point. 
    Then, in the second step, we stabilize $K$ in an arbitrary manner.
    Using Lemma~\ref{prop:lemma5} as mentioned above, we conclude that 
    $$
      E^x[m_{K}(x)] 
		\leq \mathcal{V} + E_x[m_{K}(x)].
    $$
    Note that $\mathcal{V} = C_G ( 1 + \lambda)$, as
    every time the particle visits $x$, we find a geometrically
    distributed number of sleep instructions (which are replaced by instructions $\iota$) before the particle jumps out of $x$.
    The expected number of sleep instructions found at $x$ after every visit
    is $1 + \lambda$.
    With this we obtain
    $$
        E^x[m_{K}(x)] 
		\leq  C_G ( 1 + \lambda) + E_x[m_{K}(x)]
		\leq  C_G ( 1 + \lambda) + E^x[m_{(x,K)}^1(x)],
    $$
    where the last step follows from Lemma~\ref{lem:m}.
    
    Now we establish~\eqref{eq:m} with $E$ replaced with $E_x$. 
    Using Lemma~\ref{lem:m}, we have
    $$
        E_x[m_{K}(x)] 
		\leq  E^x[m^1_{(x,K)}(x)].
    $$
    Now for the term $E^x[m^1_{(x,K)}(x)]$, let $\eta$ be an initial particle configuration having an active particle at $x$, and call $z$ the particle that starts at $x$. 
    Let $\eta_x$ be the particle configuration obtained from $\eta$ by removing $z$.
    We carry out a two-step stabilization procedure, as in the previous case. 
    In the first step, we move $z$ ignoring any sleep instruction seen until $z$ exits $K$. 
    We call $\mathcal{V}$ the expected number of topplings at $x$ up to this point. 
    In the second step, we perform a weak stabilization of $(x, K)$ starting from the particle configuration $\eta_x$.
   Note that by Lemma~\ref{prop:lemma1},  Lemma~\ref{prop:lemma5},  and Remark~\ref{remark:weakstab}, we obtain that after performing some      
   illegal topplings and ignoring some sleep instructions, 
   the value of $m^1_{(x,K)}(x)$ can only increase. Thus, we conclude that
   $
      E^x[m^1_{(x,K)}(x)] \leq \mathcal{V} + E_x[m^1_{(x,K)}(x)].
   $
   As in the previous case, we have that $\mathcal{V} = C_G ( 1 + \lambda)$.
   Putting everything together, we have 
   $$
      E_x[m_{K}(x)] 
      \leq E^x[m^1_{(x,K)}(x)]
      \leq C_G (1 + \lambda) + E_x[m^1_{(x,K)}(x)], 
   $$ 
   which concludes the proof.
\end{proof}

\section{Proof of Theorem~\ref{theo:transient_graphs}}\label{sec:proof_of_theorem_transient_graphs}
Let $L$ be a positive integer, and let $x\in V$ be a fixed vertex.
Let $B_L$ be the ball of radius $L$ centered at $x$. 
For any $y\in B_L$, let $p_y$ be the probability that a random walk starting from $y$ visits $x$ before exiting $B_L$. 
\begin{lem}\label{lem:psuminfty}
   For any vertex-transitive, transient graph, we have 
   $
      \sum_{y\in B_L}p_y \to \infty \text{ as $L\to\infty$}.
   $
\end{lem}
\begin{proof}
We can lower bound $p_y$ by $\tilde p_y$, the probability that a random walk starting from $y$ visits $x$ before exiting $B_L$ or returning to $y$. By symmetry, $\tilde p_y$ is equal to the probability that a random walk starting from $x$ visits $y$ before returning to $x$ and before exiting $B_L$.
	Therefore, $\sum_{y\in B_L} \tilde p_y$ is the expected number of vertices visited by a random walk starting from $x$ before returning to $x$ and before exiting $B_L$.
	In a transient graph, this random walk has a positive probability of never returning to $x$, in which case it visits at least $L$ vertices. This 
	establishes the lemma.
\end{proof}

\begin{proof}[Proof of Theorem~\ref{theo:transient_graphs}]
    We will stabilize $B_L$ and show that, for any fixed $\mu>0$ there exists a fixed $\lambda>0$ small enough such that
	the number of topplings at $x$ goes to infinity with $L$. This implies that $\mu_c\to 0$ as $\lambda\to 0$.
	
	Let $\eta$ be the initial particle configuration inside $B_L$ and let $\eta_s$ be the particle configuration inside $B_L$ obtained after stabilization of $B_L$.
	Then $\eta_s$ only contains sleeping particles.
	For each particle of $\eta_s$, we start a so-called \emph{ghost particle} which performs independent simple random walk steps until exiting $B_L$.
	Let $W_L$ be the number of visits to $x$ by particles or ghosts, and let $R_L$ be the number of times that $x$ was visited by ghosts. 
	So $W_L-R_L$ is the number of topplings at $x$ during the stabilization of $B_L$.
	Let 
	$N_0$ be the number of visits to $x$ of a random walk that starts from $x$ and is killed upon exiting $B_L$. 
   For simplicity, let $q=q(\lambda)=3\sqrt{\lambda\lr{C_G(1+\lambda)+1}}$, the upper bound in the second part of Theorem~\ref{theo:boundsQ}. 
   Hence,
	\begin{equation}
	   \E[W_L-R_L]
	   =\sum\nolimits_{y\in B_L}\lr{\mu-Q(y,B_L)}p_y\E[N_0]
	   \geq \lr{\mu-q}\E[N_0]\sum\nolimits_{y\in B_L}p_y.
	   \label{eq:wr2nd}
	\end{equation}
	Note that $N_0$ is a geometric random variable and, for any transient graph, it holds that $\E[N_0] < \infty$ as $L\to\infty$.
	Also, Lemma~\ref{lem:psuminfty}
	gives that for any $\mu>q$, $\E[W_L-R_L]\to\infty$ as $L\to\infty$.
	We want to show that 
	\begin{equation}
	   \PR\lr{W_L -R_L \leq \frac{\E[W_L-R_L]}{3}} \leq c<1,
	   \label{eq:goal}
	\end{equation}
	for some constant $c$ independent of $L$.
	This implies that $\liminf_{L\to\infty}\PR\lr{W_L -R_L > \frac{\E(W_L-R_L)}{3}}>0$.
	By the 0-1 law, we then obtain that $W_L -R_L$ goes to infinity almost surely, concluding the proof. 
	
	In order to establish~\eqref{eq:goal}, note that 
   \begin{align}
      \label{eq:unionbound}
      &{P}\lr{{W}_L - {R}_L \leq \frac{E [{W}_L - {R}_L]}{3}}\nonumber\\
      &= {P}\lr{ W_L  - E [{W}_L]+\frac{E [{W}_L - {R}_L]}{3}\leq {R}_L - E [{R}_L] -\frac{E [{W}_L - {R}_L]}{3}} \nonumber\\
      &\leq {P}\lr{\left| W_L  - E [{W}_L]\right|\geq \frac{E [{W}_L - {R}_L]}{3}}+ {P}\lr{\left|{R}_L - E [{R}_L]\right| \geq \frac{E [{W}_L - {R}_L]}{3}}. 
   \end{align}
   We now use Chebyshev's inequality, which gives
   \begin{equation}
      \label{eq:boundcheb}
      \begin{split}
       {P}
       \lr{{W}_L - {R}_L \leq  \frac{E [{W}_L - {R}_L]}{3}} \, & \leq  
       9 \frac{\mathrm{Var}({W}_L)}{E^2[{W}_L-{R}_L]} + 9  \frac{\mathrm{Var}(R_L)}{E^2[{W}_L-{R}_L]}.
       \end{split}
   \end{equation}
   We claim that 
   \begin{equation}
      \lim_{L\to \infty} \frac{\Var(W_L)}{\E^2[W_L-R_L]}=0,
      \label{eq:wlim}		  
   \end{equation}
   and that for any $\mu>0$ and for any small enough $\lambda$,
   \begin{equation}
	   \limsup_{L\to \infty} \frac{\Var(R_L)}{\E^2[W_L-R_L]}\leq\frac{q}{(\mu-q)^2}.
      \label{eq:rlim}		  
   \end{equation}
   Note that the above bound goes to $0$ as $\lambda\to0$. 
   Putting~\eqref{eq:wlim} and~\eqref{eq:rlim} into \eqref{eq:boundcheb} establishes \eqref{eq:goal}, which concludes the proof of the theorem.
   
   It remains to establish~\eqref{eq:wlim} and~\eqref{eq:rlim}.
	For any $3$ independent random variables $A,B,C$ note that 
	\begin{equation}
	   \Var(ABC) = \E[A^2]\E[B^2]\E[C^2] - \E^2[ A]\E^2[B]\E^2[C]. 
	   \label{eq:var}
	\end{equation}
    Then using independence we can write $\Var(W_L)=\sum_{y\in B_L}\Var(\ind{\eta(y)=1}I_yN_0)$, where $I_y$ is the indicator that a random walk 
    starting from $y$ visits $x$ before exiting $B_L$; hence, $p_y=\E[I_y]$.
   Now applying~\eqref{eq:var}, we obtain
   \begin{align*}
      \Var(W_L) 
      &= \sum_{y\in B_L}\lr{\mu p_y \E[N_0^2] - \mu^2 p_y^2 \E^2[N_0]}\\
      &= \mu \E[N_0^2]\sum_{y\in B_L}p_y\lr{1 - \mu p_y\frac{E^2[N_0]}{E[N_0^2]}}
      \leq \mu\E[N_0^2]\sum_{y\in B_L}p_y.
   \end{align*}
   Therefore, using~\eqref{eq:wr2nd},
   $$
      \frac{\Var(W_L)}{\E^2(W_L-R_L)} 
      \leq \frac{\mu \E[N_0^2]}{(\mu-q)^2\E^2[N_0]\sum_{y\in B_L}p_y}
      \to 0,
   $$
   since $\sum_{y\in B_L}p_y\to \infty$ by Lemma~\ref{lem:psuminfty}, while all the other terms are bounded away from both infinity and zero.
	
   Now we turn to~\eqref{eq:rlim}.
   For $y\in B_L$, write 
   $$
      S_y=\ind{\eta_s(y)=1}, 
      \quad
      s_y = \E[S_y]=Q(y,B_L),
      \quad\text{and}\quad
      s_{x,y} = \E[S_xS_y].
   $$
   Using this notation, we have 
   $R_L = \sum_{y\in B_L} S_y I_y N_0$.
   Since
   $\Var(R_L) = \E[R_L^2] - \E^2[R_L]$,
   we write 
   $$
      \E R_L^2 = \sum_{y\in B_L} \E\lrb{S_y I_y N_0^2} + \sum_{y, z \in B_L, y\neq z} \E\lr{S_yS_zI_yI_z N_0 N_0'},
   $$
   where $N_0,N_0'$ are independent and identically distributed. Using independence, we have
   $$
      \E R_L^2 = \sum_{y\in B_L} s_y p_y \E[N_0^2] + \sum_{y, z \in B_L, y\neq z} s_{y,z} p_y p_z \E^2[N_0].
   $$
   Hence,
   \begin{align*}
      \Var(R_L) 
      &= \sum_{y\in B_L} s_y p_y \E[N_0^2] + \sum_{y, z \in B_L, y\neq z} s_{y,z} p_y p_z \E^2[N_0] - \lr{\sum_{y\in B_L}s_yp_y \E[N_0]}^2\\
      &= \sum_{y\in B_L} \lr{s_y p_y\E[N_0^2] -s_y^2p_y^2\E^2[N_0]}+ \sum_{y, z \in B_L, y\neq z} (s_{y,z}-s_ys_z) p_y p_z \E^2[N_0]\\
      &\leq \sum_{y\in B_L} s_y p_y\E[N_0^2]+ \sum_{y,z \in B_L, y\neq z} (s_{y}-s_ys_z) p_y p_z \E^2[N_0]\\
      &\leq q\E[N_0^2]\sum_{y\in B_L} p_y+ q \E^2[N_0]\sum_{y, z \in B_L, y\neq z}  p_y p_z.
   \end{align*}
   Finally, we obtain
   \begin{align*}
      \frac{\Var(R_L)}{\E^2(W_L-R_L)}
      &\leq \frac{q\E[N_0^2]\sum_{y\in B_L} p_y+ q\E^2[N_0]\sum_{y, z \in B_L, y\neq z}  p_y p_z}{(\mu-q)^2\E^2[N_0] \lr{\sum_{y\in B_L} p_y}^2}\\
      &\leq \frac{q\E[N_0^2]}{(\mu-q)^2\E^2[N_0] \sum_{y\in B_L} p_y}
         + \frac{q}{(\mu-q)^2}.
   \end{align*}
   Note that for any fixed $\lambda>0$ the first fraction goes to $0$ with $L$ since $\sum_x p_x\to\infty$ and all the other terms are bounded away from zero and infinity. 
   The second term can be made arbitrarily small since $q\to 0$ as $\lambda\to 0$. In particular, if $\mu> q+3\sqrt{q}$, the second term is smaller than $1$, so ARW is active almost surely.
   This establishes~\eqref{eq:rlim}.
\end{proof}

\section{Proof of Theorem \ref{theo:activity_non_amenable}}\label{sec:proof_of_theorem_activity_non_amenable}
We prove Theorem~\ref{theo:activity_non_amenable} by first establishing general sufficient conditions
that give $\mu_c<1$ (Theorem~\ref{theo:suffcondnonamenable} below), and then showing that graphs of positive speed for random walks satisfy those conditions.

Let $x\in V$ be a fixed vertex of $G$, which we refer to as the \emph{origin}. 
Let $\{ X(t)\}_{t \in \mathbb{N}}$ denote a simple random walk on $G$ 
starting from the origin, and let $\{Y(t)\}_{t\in\mathbb{N}}$ be independent
random variables such that, for any $t\in\mathbb{N}$, we have $Y(t)=0$ with probability $\frac{1}{1 + \lambda}$
and $Y(t)=1$ with probability $\frac{\lambda}{1 + \lambda}$. 
Let 
$B_L$ be the ball of radius $L$ centered at $x$, and let 
$A_L$ be the vertices at distance $L$ from $x$.
For any set $V'\subset V$, let 
$$
   \tau_{V'} := \min\{t\in\mathbb{N} \, \colon \, X(t)\in V'\}
$$ 
be the first hitting time of the random walk to $V'$ and 
$$
   \tau_{V'}^+ := \min\{t\geq 1 \, \colon \, X(t)\in V'\}
$$ 
be the first return time of the random walk to $V'$.
Finally, let 
$$
   \tau_{V'}^\mathrm{k} := \min\{t\geq 1 \, \colon \, X(t)\notin V' \text{ and } Y(t)=1\}.
$$ 
We can interpret the above quantity by considering that the random walk is ``killed'' outside $V'$ at times $t$ when $Y(t)=1$; using this, 
$\tau_{V'}^\mathrm{k}$ gives the time the random walk is killed.

Here we consider that the initial particle configuration, denoted by $\eta$, is given by any product of identical
measures on $\mathbb{N}^V$ with density $\E[\eta(x)]=\mu$. We assume that the graph $G$
is vertex-transitive. Note that the assumption of positive speed of the random walk on $G$ is not
required in the next theorem.
Let $\nu_0=\PR(\eta(x)=0)$ be the probability that a vertex is empty
at time $0$, which is the same for all vertices.
\begin{theorem}\label{theo:suffcondnonamenable}
   Given positive integers $n<L$, set $\Lambda=V \setminus B_{L}$ and let 
   $$
      N_n^L := |\{t\in \mathbb{N} \,\colon\, X(t) \in A_n \text{ and } t< \tau_{\Lambda} \land \tau^+_{x}\}|
   $$ 
   be the number of visits of $X(t)$ to $A_n$
   before $X(t)$ enters $\Lambda$ or returns to the origin. 
   Let 
   $$
      \tilde{N}_n^L := |\{t\in \mathbb{N} \,\colon\, X(t) \in A_n \text{ and } t< \tau_{\Lambda} \land \tau^+_{x} \land \tau^\mathrm{k}_{B_{n-1}}\}|
   $$
   be the number of visits of $X(t)$ to $A_n$ before $X(t)$ enters $\Lambda$, returns to the origin
   or is ``killed'' outside $B_{n-1}$.
   Let also $M_L := \sum_{n=0}^L{N_n^L}$ and 
   $\tilde{M}_L := \sum_{n=0}^L{\tilde{N}_n^L}$.
   If given $\mu$ and $\lambda$ we have
   \begin{equation}
      \liminf_{L \rightarrow \infty}  \frac{E [ \tilde{M}_L ]}{E[M_L]} > \frac{\nu_0}{\mu+\nu_0}.
      \label{eq:suffcond}
   \end{equation}
   then ARW is active almost surely.
\end{theorem}
\begin{proof}
   We will define a stabilization procedure for $B_{L}$ and show that the number of topplings at the origin goes to infinity with $L$.
   We will do the stabilization by moving particles located at different levels step by step. At the first step
   we move all particles which are located in $A_L$, at the next step we move all particles which are located in $A_{L-1}$, and so on. 
   The same particle might be moved several times in the course of the whole procedure. We now define such steps.
   
   \vspace{-0.5cm}
  \paragraph{First step.} Let $\eta$ be the initial particle configuration, and let $Z_L$ be the particles of $\eta$ which are in $A_L$. Order the particles in $Z_L$ in some arbitrary manner.
   Consider the first particle in the order and move that particle until one of the following events occur:
   \begin{enumerate}
   \itemsep0.5em 
   \item the particle reaches the origin,
   \item\label{it:empty} the particle reaches an empty site in $B_{L-1}$,
   \item\label{it:sleep} the particle ``uses'' a sleep instruction in $V \setminus B_{L-1}$,
   \item the particle reaches $\Lambda$.
   \end{enumerate}
   Then, take the second particle in the order and move it several times until one of the four events above occurs.
   After that, take the third particle in the order and do the same. 
   Repeat this procedure until all particles of $Z_L$ have been moved. We obtain a new particle configuration that we denote by $\eta^1$.
   
      \vspace{-0.5cm}
   \paragraph{Second step.} Let $Z_{L-1}$ be the particles of $\eta^1$ which are in $A_{L-1}$. 
   Note that $Z_L$ and $Z_{L-1}$ are not necessarily disjoint, since particles in $Z_L$ could have ended in $A_{L-1}$ after they were moved in the first step. 
   Order the particles of $Z_{L-1}$ in some arbitrary order.
   Now move the first particle in the order of $Z_{L-1}$ until one of the following events occur:
    \begin{enumerate}
   \itemsep0.5em 
   \item the particle reaches the origin,
   \item the particle reaches an empty site in $B_{L-2}$,
   \item the particle ``uses'' a sleep instruction in $V \setminus B_{L-2}$,
   \item the particle reaches $\Lambda$.
   \end{enumerate}
   Then, take the second particle in the order and move it several times until one of the four events above occurs. 
   After that, take the third particle in the order and do the same. 
   Repeat the same procedure until all particles of $Z_{L-1}$ have been moved. 
   We obtain a new particle configuration that we denote by $\eta^2$.
   
   \vspace{-0.5cm}
   \paragraph{Next steps.} We repeat the procedure above, analogously defining the set of particles $Z_{L-i}$, $i\in\{2,3,4,...,L-1\}$, and obtaining the particle configuration $\eta^{i+1}$.
   
   Note that $\eta^L$ may not be a stable configuration. However, letting $G_L$ be the total number of particles that stop at the origin during the procedure described above, 
   the Abelian property (Lemma~\ref{prop:lemma2}) implies that $m_{B_L}(x) \geq G_L$.
   Our goal is to prove that there exists a constant $c>0$ independent of $L$ such that the probability that $G_L > c L$ is bounded away from $0$, which implies that ARW is active by the 0-1 law.
   In order to estimate $G_L$, we introduce \textit{ghost particles} as in Section~\ref{sec:proof_of_theorem_transient_graphs}.
   Ghost particles can be created at any step of our procedure. 
   Consider the $(L-n+1)$th step, where we move particles from the set $Z_n$, the set of particles of $\eta^{L-n}$ that are located in $A_n$. 
   Let $w$ be one of the particles that is moved at this step. 
   Let $z \in A_n$ be its starting vertex. We create a ghost particle if the two next conditions hold:
   \begin{enumerate}[(i)]
   \itemsep0.5em 
     \item  $\eta(z)=0$ (i.e., $z$ is empty for the initial particle configuration),
      \item the motion of $w$ stops because it ``uses'' a sleep instruction at some site $y \in V \setminus B_{n-1}$ (i.e., the motion of $w$ stops due to condition~\ref{it:sleep} in the procedure above).
   \end{enumerate}
   The ghost particle is then created at  $y \in V \setminus B_{n-1}$, the site where the particle $w$ uses the sleep instruction.
   We call the site $z \in A_n$ above \textit{the site that is associated to the ghost}.  
   A crucial point to observe is that, in order for $w$ to create a ghost during step $L-n+1$, 
   it is necessary that $w$ is in $V \setminus B_n$ for the initial particle configuration $\eta$, 
   and that at some previous step $w$ is moved until reaching the site $z$, which was empty at that time (so $w$ is stopped according to condition~\ref{it:empty} in the procedure above). 
   Note that every particle creates at most one ghost in the course of the whole procedure. 
   Indeed, when this happens, the particle that is responsible for the generation of the ghost is not moved any more at any subsequent step. 
   After being created, each ghost particle performs independent simple random walk steps until reaching $\Lambda \cup\{x\}$, when it then stops.

   Let $W_L$ be the number of particles and ghosts visiting the origin, and let 
   $R_L$ be the number of ghosts visiting the origin.
   Then,
   $$
   G_L \, = \,  W_L - R_L.
   $$
   We now estimate the terms $W_L$ and $R_L$ separately.
   For any $j \in \mathbb{N}$ and $z \in V$, let
   $\left( X^{(z,j)}(t) \right)_{t \in \mathbb{N}}$ be an independent random walk
   on $V$ starting from $z$ and  
   $\left( Y^{(z,j)}(t) \right)_{t \in \mathbb{N}}$ be an infinite sequence of i.i.d. random variables
   such that $Y^{(0,0)}(0) = 1$ with probability $\frac{\lambda}{1+\lambda}$
   and $Y^{(0,0)}(0) = 0$ with probability $\frac{1}{1+\lambda}$.
   Let $\tau_S^{(z,j)}$  be the first time the random walk $\left( X^{(z,j)}(t) \right )_{t \in \mathbb{N}}$ visits
   the set $S \subset V$ and let us write simply $\tau_y^{(z,j)}$ if $S = \{y\}$.
   Then, 
   \begin{equation}
   \label{eq:WL}
      W_L \text{ stochastically dominates }
      \tilde{W}_L  := \sum\limits_{{n = 1} }^L   \sum\limits_{\, z \in A_n \,  }   \sum\limits_{ j = 1   }^{\eta(z) }   
   \ind{\mathcal{A}^{(z,j)} \cap \mathcal{B}^{(z,j)}},
   \end{equation}
   where $\mathcal{A}^{(z,j)}  := \{ \tau_{x}^{(z,j)} < \tau^{(z,j)}_{\Lambda} \}$,
   $\mathcal{B}^{(z,j)}  := \{ Y^{(z,j)}(t)=0 $ for any $t \leq \tau^{(z,j)}_{x}$ such that
   $X^{(z,j)}(t) \notin  B_{|z|-1} \}$,  
   and $|z|$ denotes the distance between $z$ and $x$. 
   
   Now we make a {crucial observation} for the estimation of $R_L$. Recall that every ghost can be associated to the site where the particle starts at the step it uses the sleep instruction and generates that ghost. From the definition of our procedure, it follows that for every site $z \in B_L$ such that  $\eta(z)=0$, \textit{there exists at most one ghost that can be associated to $z$}. It also follows that if $z \in B_L$ is such that $\eta(z)=1$, then no ghost can be associated to that site. Thus, if from every site $z\in B_L$ with $\eta(z)=0$ we start a \textit{sleeping random walk}
   $ \left( \,  X^{(z,0)}(t),\,  Y^{(z,0)}(t) \, \right)_{t \in \mathbb{N}}$
   and we count $\tilde{R}_{L}$, the number  
   of them which hit the origin before entering $\Lambda$ and such that $Y(t)=1$ 
   somewhere in $V \setminus B_{|z|-1}$,
   we conclude that
   \begin{equation}
   \tilde{R}_L \text{ stochastically dominates } R_L.
   \end{equation}
   Hence, we write,
   \begin{equation}
   \label{eq:RL}
    \tilde{R}_L  := \sum\limits_{{n = 1} }^L   \sum\limits_{\, z \in A_n \,  }   
    \ind{\eta(z)=0}
   \cdot \ind{\mathcal{A}^{(z,0)} \cap \overline{\mathcal{B}}^{(z,0)}} ,
   \end{equation}
   where for clarity we denote by
   $\overline{\mathcal{B}}^{(z,0)} := \lr{\mathcal{B}^{(z,0)}}^\compl = \{ Y^{(z,0)}(t)=1 $ for some $t \leq \tau^{(z,0)}_{x}$ such that
   $X^{(z,0)}(t) \notin B_{|z|-1} \}$ the complement
   of ${\mathcal{B}}^{(z,0)}$.
   As the initial particle configuration is distributed according to a product measure,
   from (\ref{eq:WL}) and (\ref{eq:RL}) it follows that,
   \begin{equation}
      \begin{split}
      \label{eq:expWLRL}
      E[\tilde{W}_L] - E[\tilde{R}_L]  & = 
      \sum\limits_{{n = 1} }^L   \sum\limits_{\, z \in A_n \,  }    \lrb{ \mu\, \cdot \,  P\left( \mathcal{A}^{(z,0)} \cap {\mathcal{B}}^{(z,0)}
      \right) \, - \,  \nu_0 \, \cdot \,  P \left( \mathcal{A}^{(z,0)} \cap \overline{\mathcal{B}}^{(z,0)} \right)} \\
      &= \sum\limits_{{n = 1} }^L   \sum\limits_{\, z \in A_n \,  }    \lrb{  (\mu+\nu_0)P\left( \mathcal{A}^{(z,0)} \cap {\mathcal{B}}^{(z,0)}
      \right) \, - \,  \nu_0 \, \cdot \,  P \left( \mathcal{A}^{(z,0)} \right)}.
      \end{split}
   \end{equation}
   To simplify the notation, we will henceforth drop the $0$'s from the superscript in the terms above.
   When analyzing the term $P(\mathcal{A}^z \cap \mathcal{B}^z)$, consider the last time $t$ that the random walk starting from $z\in A_n$ visits 
   $A_n$ before reaching the origin. We will denote by $y$ the vertex of $A_n$ where the random walk is in its last visit to $A_n$.
   Hence, decomposing in $y$ and $t$, we have
   \begin{equation}
      \label{eq:PABsplit}
      P(\mathcal{A}^z \cap \mathcal{B}^z) = \sum\limits_{y \in A_n} \sum\limits_{t = 1}^{\infty} P( \mathcal{C}^{z,y,t} \cap \mathcal{D}^{z,t} \cap \mathcal{E}^{z,t}) \cdot P(\tau^y_x < \tau^y_{A_n,+}),
   \end{equation}
   where $\mathcal{C}^{z,y,t} := \{ X^{z}(t)=y\}$,
   $\mathcal{D}^{z,t} := \{Y^z(t)=0 \, $ for any $i \leq t$ such that $X^z(i) \notin B_{|z|-1}\}$,
   $\mathcal{E}^{z,t}:= \{ \tau^z_{ \{x\} \cup \Lambda} > t \}$,
   $\tau^z_{S,+}$ is the first return time of the random walk starting from $z$ to the set $S \subset V$.
   Now since graph is transitive, any path of a random walk from a vertex $z_1$ to $z_2$ occurs with the same probability as the reversed path for
   a random walk going from $z_2$ to $z_1$.
   This gives that, for $y,z \in A_n$,
   \begin{equation}
   \label{symm1}
   P(\tau_x^y < \tau^y_{A_n,+}) = P ( \{ X^x(\tau_{A_n}^x) = y\} \cap \{\tau^x_{A_n} < \tau^x_{x,+}\});
   \end{equation}
   that is, the event $\tau_x^y < \tau^y_{A_n,+}$ is equivalent to the event that a random walk starting from $x$ visits 
   $A_n$ before returning to $x$, and visits $A_n$ for the first time at $y$.
   Also, for $y,z \in A_n$ and $t \in \mathbb{N}$,
   \begin{equation}
   \label{symm2}
   P( \mathcal{C}^{z,y,t} \cap \mathcal{D}^{z,t} \cap \mathcal{E}^{z,t}) = P( \mathcal{C}^{y,z,t} \cap \mathcal{D}^{y,t} \cap \mathcal{E}^{y,t}).
   \end{equation}
   Now plug (\ref{symm1}) and (\ref{symm2}) into (\ref{eq:PABsplit}).
   Summing over $z \in A_n$
   first and then over $y$ and $t$, and using the Markov property for the random walk,
   we conclude that
   \begin{multline}
   \label{eq:derivedAB}
   \sum\limits_{\, z \in A_n \,  }     \,   P\left( \mathcal{A}^{z} \cap {\mathcal{B}}^{z}  \right)= \\
     \sum\limits_{\, y \in A_n \,  } 
   \sum\limits_{t=0}^{\infty} 
   E \left[  \sum\limits_{z \in A_n}  \, \ind{\mathcal{C}^{y,z,t} \cap \mathcal{D}^{y,t} \cap \mathcal{E}^{y,t}} \right] \, \cdot \, P( \{ X^x( \tau_{A_n}^x ) = y \} \cap \{ \tau^x_{A_n} < \tau^x_{x,+} \} ) ]  =
   E [ \, \,  \tilde{N}_n^L \, \,   ].
   \end{multline}
   Similarly to (\ref{eq:derivedAB}), we obtain
   \begin{equation}
   \label{eq:derivedA}
   \sum\limits_{\, z \in A_n \,  }   \,   P\left( \mathcal{A}^{z}   \right)= 
   E [ \, \, {N}^L_n \, \, ].
   \end{equation}
   Hence, plugging (\ref{eq:derivedAB}) and (\ref{eq:derivedA}) into (\ref{eq:expWLRL}), we have
   \begin{align}
      \label{eq:expWLRLfinal1}
      E[\tilde{W}_L]-E[\tilde{R}_L] 
      &= \lr{\sum\limits_{n=0}^{L} (\mu+\nu_0)E[\tilde N_n^L]-\nu_0 E[N_n^L]} \nonumber\\
      &= (\mu+\nu_0)E[ \tilde M_L ] - \nu_0 E[ M_L ].
   \end{align}
   Note now that $E[M_L]$ goes with $L$ to the expectation of the return time of the random walk, which is infinite on any infinite,
   connected graph (see for example Theorem 1.1 in \cite{Gurevich}).
   Then, from our assumption in  (\ref{eq:suffcond}), it follows that  the lower bound above
   diverges with $L$.
   It remains to prove that
   this implies that  $G_L \rightarrow \infty$ with $L$ with positive probability,
   which in turn implies that ARW is active almost surely
   by the 0-1 law (Lemma~\ref{prop:lemma4}). 
   For this, we use the same derivation as in~\eqref{eq:unionbound} and~\eqref{eq:boundcheb}, which gives that
   \begin{align}
      {P}\lr{\tilde{W}_L - \tilde{R}_L < \frac{E [\tilde{W}_L - \tilde{R}_L]}{3}}
      &\leq 9 \frac{\Var(\tilde{W}_L)}{\E^2[\tilde{W}_L-\tilde{R}_L]} + 9  \frac{\Var(\tilde{R_L})}{\E^2[\tilde{W}_L-\tilde{R}_L]}\nonumber\\
      &\leq 9 \frac{\E[\tilde{W}_L]}{\E^2[\tilde{W}_L-\tilde{R}_L]} + 9  \frac{\E[\tilde{R_L}]}{\E^2[\tilde{W}_L-\tilde{R}_L]},
      \label{eq:boundvar}
   \end{align}
   where in the last step we use that $\Var(\tilde{W}_L)\leq \E[\tilde{W}_L]$ and $\Var(\tilde{R}_L)\leq \E[\tilde{R}_L]$ since 
   $\tilde W_L$ and $\tilde R_L$ are defined as a sum of independent Bernoulli random variables.
   Note that~\eqref{eq:expWLRLfinal1} and~\eqref{eq:suffcond} imply that 
   \begin{align}
      \E[\tilde W_L-\tilde R_L] 
      &> K \, \E[M_L] \quad\text{for some constant $K>0$ and all large enough $L$}.
      \label{eq:k}
   \end{align}
   Hence we obtain that $\E[\tilde W_L]\geq \E[\tilde R_L] $ for all large enough $L$.
   In addition, from the derivation of~\eqref{eq:expWLRL} and~\eqref{eq:expWLRLfinal1} we have
   $$
      \E[\tilde W_L]\leq (\mu+\nu_0) \E[\tilde M_L].
   $$
   Using these facts, we obtain
   \begin{align*}
      \E[\tilde W_L+\tilde R_L] 
      \leq 2\E[\tilde W_L] 
      \leq 2(\mu+\nu_0) \E[\tilde M_L].
   \end{align*}
   Plugging this into~\eqref{eq:boundvar}, and using~\eqref{eq:k}, we get
   \begin{align*}
       {P}
       \lr{\tilde{W}_L - \tilde{R}_L <  \frac{E [\tilde{W}_L - \tilde{R}_L]}{3}} 
        \leq \frac{18(\mu+\nu_0)\E[\tilde {M}_L] }{K^2 \E^2[M_L]}
        \leq \frac{18(\mu+\nu_0)}{K^2 \E[M_L]}.
   \end{align*}
   The last term converges to $0$ with $L$. 
   Hence, $W_L - R_L > \frac{E [W_L - R_L]}{3}$ is bounded away from $0$ and this concludes the proof.
\end{proof}

\begin{proof}[Proof of Theorem \ref{theo:activity_non_amenable}]
   We show that for any $\lambda>0$ 
   and $\mu > 1 - \frac{\alpha \delta}{1 + \lambda}$
   the condition in \eqref{eq:suffcond} is satisfied.  
   Observe that, conditioning on the non-return of the random walk to the origin,
   $\tilde{N}_n^L$ is stochastically larger than a random variable which takes 
   value $1$ with probability $\frac{1}{1+\lambda}$ and 
   $0$ with probability $\frac{\lambda}{1+\lambda}$,
   as the random walk hits $A_n$ at least one time.
   Hence,  
   \begin{equation}
   \label{eq:derivedAB2}
    \E[ \tilde{N}_n^L \, ]
   \geq \frac{\delta}{1 + \lambda}
   \quad\text{and, consequently,}\quad
    \E[ \tilde{M}_L \, ]
   \geq \frac{\delta}{1 + \lambda} L.
   \end{equation}
   Our goal is to 
   use that the random walk has a positive speed $\alpha$, which is to say that
   \begin{equation}
      \lim\limits_{t \rightarrow \infty} \frac{| X(t)| }{t} = \alpha \quad\text{almost surely},
      \label{eq:posspeed}
   \end{equation}
   to show that, for any $\epsilon>0$, there exists $L_0=L_0(\epsilon)$ large enough so that 
   \begin{equation}
      E[M_L] \leq \frac{L}{\alpha-\epsilon} \quad\text{for all $L\geq L_0$}.
      \label{eq:derivedA2}
   \end{equation}
   To do this, note that~\eqref{eq:posspeed} implies that, for any $\xi>0$, there exists $t_0$ large enough so that 
   $$
      P\left(X(t) > (1-\xi)\alpha t \quad \text{for all $t\geq t_0$}\right)>1-\xi.
   $$
   Now let $\Delta_L = \lceil\frac{L}{(1-\xi)\alpha}\rceil$, so for $L$ large enough we have that $X(\Delta_L)$ is outside $B_L$ with probability at least $1-\xi$. 
   If that does not happen, then $X(\Delta_L)$ is at some random vertex $y\in B_L$. Then after an additional time of $\Delta_{2L}$, with probability at least $1-\xi$, the walker 
   exits the ball of radius $2L$ centered at $y$; consequently, it also exits $B_L$. This gives that $P(X(\Delta_{3L}) \in B_L) \leq \xi^2$. Iterating this argument, we have 
   \begin{align*}
      E[M_L] 
      = \sum_{i\geq 1} P(M_L \geq i) 
      &\leq \Delta_L + \sum_{k=1}^\infty \left(\Delta_{(2k+1)L}-\Delta_{(2k-1)L}\right)\cdot P(X(\Delta_{(2k-1)L}) > L)\\
      &\leq \Delta_L + \sum_{k=1}^\infty \left(1+\frac{2L}{(1-\xi)\alpha}\right)\xi^k\\
      &= \Delta_L + \left(1+\frac{2L}{(1-\xi)\alpha}\right)\frac{\xi}{1-\xi}.
   \end{align*}
   Then~(\ref{eq:derivedA2}) follows by taking $\xi$ small enough with respect to $\epsilon$.
   Hence, we conclude that for all $L$ large enough,
   \begin{align*}
     \frac{\E[\tilde M_L]}{\E[M_L]}
     \geq \frac{\delta (\alpha-\epsilon)}{1+\lambda}.
   \end{align*}
   Thus, the condition in \eqref{eq:suffcond} is satisfied
   when $\nu_0=1-\mu$ 
   as long as $\mu > 1 - \frac{\alpha \delta}{1 + \lambda}$.
\end{proof}

\section{Proof of Theorem \ref{theo:mu_c_positive_amenable}}
\label{sec:proof_of_theorem_mu_c_positive_amenable}
\begin{proof}[Proof of Theorem \ref{theo:mu_c_positive_amenable}]
	Since $G$ is amenable and vertex transitive, we can take a sequence of subsets $\{V_n\}_{n\geq1}$ of $V$ such that 
	$V_n \to V$ as $n\to\infty$, there exists a vertex $x\in\bigcap_{n=1}^\infty V_n$, and 
	$$
	   \frac{|\partial V_n|}{|V_n|} \text{ is non-increasing and goes to $0$ as $n\to\infty$},
	$$
	where $\partial V_n$ denotes the external boundary of $V_n$; that is, the set of vertices in $V\setminus V_n$ that have an edge incident to $V_n$.
	Let $B_K$ be the ball of radius $K$ centered at $x$, and recall that $m_{B_K}(x)$ is the number of instructions used at $x$ to stabilize $B_K$.
	If we assume that $\mu>\mu_c$, then the 0-1 law (Lemma~\ref{prop:lemma4}) implies that $\Pr\lr{m_{B_K}(x)\geq 1}\to 1$ as $K\to\infty$.
	By monotonicity of this probability, for any fixed $\epsilon>0$, we can find $K=K(\epsilon)$ large enough such that 
	$$
	   \Pr\lr{m_{B_K}(x)\geq 1} \geq 1-\epsilon.
	$$
	
	For any set $V_n\subset V$, let $V_n^K$ be the set obtained by taking the union of balls of radius $K$ centered at each vertex of $V_n$. 
	Hence $V_n^K\supset V_n$. Let $N_{n,K}$ be the number of particles inside $V_n^K$ prior to the stabilization of $V_n^K$ and let $N^s_{n,K}$ be the number 
	of sleeping particles in $V_n^K$ after the stabilization of $V_n^K$. Clearly, $N^s_{n,K}\leq N_{n,K}$ almost surely.
	Let $d$ denotes the degree of each vertex of $G$; so $B_K$ has at most $d^K$ vertices. Note that 
	$$
	   \E[N_{n,k}] 
	   = |V_n^K| \mu
	   \leq \lr{|V_n| + d^K |\partial V_n|}\mu
	   = \lr{1 + d^K \frac{|\partial V_n|}{|V_n|}} |V_n| \mu.
	$$
	Also, from \eqref{eq:lbound} in Theorem~\ref{theo:boundsQ}, we have
	$$
	   \E[N_{n,K}^s]
	   \geq \sum_{y \in V_n} Q(y,V_n^K)
	   \geq \lr{\frac{\lambda}{1+\lambda}}\sum_{y \in V_n} \Pr\lr{m_{V_N^K}(y)\geq 1}.
	$$
	Since $V_N^K$ contains a ball of radius $K$ centered at $y$, by monotonicity and transitivity we obtain
	$$
	   \E[N_{n,K}^s]
	   \geq |V_n|\lr{\frac{\lambda}{1+\lambda}}\sum_{y \in V_n} \Pr\lr{m_{B_K}(x)\geq 1}
	   \geq |V_n| \lr{\frac{\lambda}{1+\lambda}}\lr{1-\epsilon}. 
   $$
	Since $\E[N_{n,k}] \geq \E[N_{n,k}^s]$, placing the two inequalities together yields
	$$
	   \mu \geq \lr{\frac{\lambda}{1+\lambda}}\lr{1-\epsilon}\lr{1 + d^K \frac{|\partial V_n|}{|V_n|}}^{-1}.
	$$
	Now set $n=n(d,K)$ large enough such that $\mu\geq \lr{\frac{\lambda}{1+\lambda}}\lr{1-\epsilon}^2$. 
	Therefore, assuming that $\mu>\mu_c$ implies that $\mu\geq \lr{\frac{\lambda}{1+\lambda}}\lr{1-\epsilon}^2$, which completes the proof since $\epsilon>0$ is arbitrary.
\end{proof}

\section{Proof of Theorem \ref{theo:tree}}\label{sec:proof_of_theo_tree}
In order to show that $\mu_c>0$ for any $\lambda>0$ when $G$ is a $d$-regular tree, we will relate an stabilization procedure to a certain branching process in $\mathbb{Z}$. 
To avoid ambiguity we will refer to the particles of the branching process as \emph{tokens}. Initially the branching process starts with $d$ tokens at position $1$.
We will show that $\mu_c>0$ holds if with positive probability we have that no token ever visits a position $k\leq 0$.

We start defining this branching process.
Start with $d$ tokens at position $1$. The process evolves in discrete steps, where at each step we update the position of each token independently. 
Given a token at position $k\in\mathbb{Z}$, we update it as follows. 
With probability $\alpha$, the token advances one position, jumping to position $k+1$. 
With probability $1-\alpha$, there is a \emph{branching}. This means that the token is deleted and is replaced by $\ell d$ tokens at position $k-\ell$, 
where $\ell\geq 1$ is an independent geometric random variable of success probability $\beta$; i.e., $P(\ell = z)=(1-\beta)^{z-1}\beta$.
The value $d$ here will later be the same as the degree of the 
tree, that is why we choose to use the same letter, while $\alpha$ and $\beta$ are some additional parameters that will be related to $\mu$ and $\lambda$.

\begin{lem}\label{lem:branch}
   For all $d\geq1$ and all $\beta\in(0,1)$, there exists $\alpha_0\in(0,1)$ large enough so that, for all $\alpha \in (\alpha_0,1]$, with positive probability there will never be a token in positions $k\leq0$.
\end{lem}
\begin{proof}
   Let $\gamma = \sqrt{1-\beta}$.
   At time $t$, if $k_1,k_2,\ldots$ are the positions of the tokens at that time, define the function
   $$
      \Psi_{t} = \sum_{i\geq 1}\gamma^{k_i}.
   $$
   Let $R$ be the smallest value such that $d\gamma^{R+1}<1/5$, and consider the event 
   \begin{equation}
      \mathcal{E} = \left\{\text{in the first $R$ steps all tokens advance and do not branch}\right\}.
   \end{equation}
   Note that $P(\mathcal{E})=\alpha^{dR}$, and $\mathcal{E}$ implies that 
   at time $R$ all tokens are at position $R+1$. Thus,
   $$
      P(\Psi_R < \tfrac{1}{5}) \geq P(\mathcal{E}) = \alpha^{dR}.
   $$
   Note that if any token reaches a position $k\leq 0$, then we have $\Psi_t\geq 1$. 
   We show that with positive probability $\Psi_t<1$ for all $t\geq R$.
   For this it suffices to show that $\Psi_t$ is a supermartingale, provided $\alpha$ is large enough. 
   
   Let $\mathcal{F}_t$ denote the
   filtration given by the position of the tokens at times $0,1,\ldots,t$. 
   Now we compute the change in $\Psi_t$ in one step. 
   Define $\Delta_k$ as the expected change in $\Psi_t$ caused by moving a token from position $k$, assuming that there is at least one token at position $k$ at that time. 
   By the form of $\Psi_t$, we have that $\Delta_k$ does not depend on $t$.
   Given a token at position $k$, since the token advances one position with probability $\alpha$, and gets replaced by $\ell d$ tokens at position $k-\ell$ with probability
   $(1-\alpha)(1-\beta)^{\ell-1}\beta$, we have that 
   \begin{align*}
      \Delta_k
      &=-\gamma^k + \alpha \gamma^{k+1} + (1-\alpha)\sum_{\ell=1}^\infty \ell d \beta (1-\beta)^{\ell-1} \gamma^{k-\ell}\\
      &=-\gamma^k\left(1 - \alpha \gamma - \frac{d\beta (1-\alpha)}{1-\beta}\sum_{\ell=1}^\infty \ell \left(\tfrac{1-\beta}{\gamma}\right)^{\ell}\right).
   \end{align*}
   Since $\frac{1-\beta}{\gamma}=\gamma<1$, the sum above converges to $\frac{\gamma}{(1-\gamma)^2}$. 
   This and replacing $1-\beta$ with $\gamma^2$ yield
   $$
      \Delta_k
      =-\gamma^k\left(1 - \alpha \gamma - \frac{d(1-\gamma^2)(1-\alpha)}{\gamma (1-\gamma)^2}\right)
      =-\gamma^k\left(1 - \alpha \gamma - \frac{d(1+\gamma)(1-\alpha)}{\gamma (1-\gamma)}\right).
   $$
   Hence,
   \begin{align*}
      \E\lr{\Psi_{t+1} \mid \mathcal{F}_t}
      = \Psi_t + \sum_{i\geq1}\Delta_{k_i}
      &= \Psi_t - \sum_{i\geq1}\gamma^{k_i}\left(1 - \alpha \gamma - \frac{d(1+\gamma)(1-\alpha)}{\gamma (1-\gamma)}\right)\\
      &= \Psi_t \left(\alpha \gamma + (1-\alpha)\frac{d(1+\gamma)}{\gamma (1-\gamma)}\right).
   \end{align*}
   Note that we can make the term inside the parenthesis as close to $\gamma$ as possible by having $\alpha$ close to $1$. 
   So, since $\gamma<1$, by having $\alpha$ close enough to $1$ we obtain that $\{\Psi_t\}_t$ is a supermartingale.
   Let $\tau$ be the first time that $\Psi_{R+\tau}\geq 1$ and $n$ be any positive integer. 
   We apply the optional stopping theorem for the almost surely bounded stopping time $\tau\land n$, and obtain under the event $\{\Psi_R< 1/5\}$ that 
   $$
      \frac{1}{5} > \Psi_R \geq E\lr{\Psi_{R+(\tau \land n)}} \geq E\lr{\Psi_{R+\tau} \mid \tau \leq n} P\lr{\tau \leq n} \geq P\lr{\tau \leq n}.
   $$
   Since $n$ is arbitrary, the probability that there will never be a token in positions $k\leq 0$ is at least $\alpha ^{dR} \frac{4}{5}$. 
\end{proof}

Before proceeding to the proof of Theorem~\ref{theo:tree}, we state a well-known lemma regarding random walks on regular trees. 
\begin{lem}\label{lem:pl}
   For any $\ell$, let $p_\ell$ be the probability that a random walk starting at distance $\ell$ from the origin ever 
   visits the origin. Then, for a $d$-regular tree we have 
   $
      p_\ell = \lr{\frac{1}{d-1}}^\ell.
   $
\end{lem}
\begin{proof}
   The lemma follows by checking that if we set $p_\ell=a^\ell$ for some $a>0$, then $a=\frac{1}{d-1}$ 
   is the only solution in $(0,1)$ of the recursion 
   $a^\ell = \frac{1}{d}a^{\ell-1} + \lr{\frac{d-1}{d}}a^{\ell+1}$.
\end{proof}

\begin{proof}[Proof of Theorem~\ref{theo:tree}]
   Since a simple random walk in a $d$-regular tree has positive speed, Theorem~\ref{theo:activity_non_amenable} gives that $\mu_c<1$ for any $\lambda>0$.
   Also, since a $d$-regular tree with $d\geq3$ is a transient graph, Theorem~\ref{theo:transient_graphs} gives that $\lim_{\lambda\downarrow 0}\mu_c=0$.
   It remains to show that $\mu_c>0$ for any $\lambda>0$.   
   
   For any $\rho\in(0,1)$, let $\mathrm{Geo}_{\rho}$ be a geometric random variable of success probability $\rho$.
   We assume that, at each vertex $v$, the number of particles initially located at $v$ is distributed independently according to the distribution of $\mathrm{Geo}^*_{1-\mu}=\mathrm{Geo}_{1-\mu}-1$. 
   This is enough for our purposes, because $P(\mathrm{Geo}^*_{1-\mu}=0)=1-\mu$, so $\mathrm{Geo}^*_{1-\mu}$ stochastically dominates a Bernoulli random variable of mean $\mu$.
   Then using monotonicity of ARW (cf.\ Lemma~\ref{prop:lemma3}), if for a given $\mu>0$ ARW almost surely fixates starting from this initial configuration of particles, 
   ARW almost surely fixates starting from a Bernoulli field of particles of density $\mu$. This will establish that $\mu_c>0$. 

   We will employ a beautiful stabilization procedure developed by Rolla and Sidoravicius~\cite{Rolla} for the one-dimensional lattice $\mathbb{Z}$. 
   We will need to carry out a much more delicate analysis for the case of a $d$-regular tree. 
   Let $x_1,x_2,\ldots$ be the particles ordered according to their initial
   distance to the origin, with $x_1$ being the closest particle to the origin. 
   Let $L$ be an arbitrarily large integer, and consider the finite system inside $B_L$, the ball of radius $L$ around the origin.
   Our goal is to show that with positive probability we stabilize $B_L$ without any particle visiting the origin.
   The idea is to move particles in order, ignoring some sleep instructions and stopping them when they see a sleep instruction near the origin. We do this to 
   pack the particles as close as possible to the origin in such a way that the gap between the particles that have already been moved and the particles that have not yet been moved 
   increases with time. This creates more room for particles to fixate, allowing the stabilization procedure to be carried out until the end without activating any particle that was moved before.
   
   Now we describe the stabilization procedure in details. We will define sets $C_k\subset V$, $k\geq 0$. 
   Let $C_0$ consist of only the origin. 
   We start by moving particle $x_1$ repetitively, ignoring all sleep instructions, until it either reaches $V\setminus B_L$ or $C_0$.
   If it reaches $V\setminus B_{L}$, then we set $C_1=C_0$. 
   Otherwise, let $z_1,z_2,\ldots,z_T\in V$ be the sequence of vertices visited by $x_1$, with $z_1$ being the initial location of $x_1$ and $z_T\in C_0$. 
   Define $\tau$ to be the largest integer so that $x_1$ ignored a sleep instruction at $z_\tau$; if $x_1$ never ignored a sleep instruction
   until it reaches the origin, we declare that the procedure failed. If the procedure has not failed,
   set $C_1 = C_0 \cup \{z_\tau,z_{\tau+1},\ldots,z_{T-1}\}$. 
   Using the terminology in~\cite{Rolla}, we see $C_1$ as the set of \emph{corrupted} vertices after $x_1$ is moved. 
   If after defining $C_1$ we have that at least one of the subsequent particles $x_2,x_3,\ldots$ is located inside $C_1$, we declare that the procedure fails.
   
   The idea behind the definition of $C_1$ is the following. 
   We would like to move $x_1$ as close as possible to $C_0$, to the point that we stop $x_1$ at the last sleep instruction it sees before visiting $C_0$. However, in order to observe that $z_\tau$ is the 
   vertex where the last sleep instruction is seen by $x_1$, we need to observe the instructions at $z_{\tau+1},z_{\tau+2},\ldots,z_{T-1}$. 
   This corrupts the array of instructions at the vertices $z_{\tau+1},z_{\tau+2},\ldots,z_{T-1}$, so we cannot use these arrays of instructions when we move the subsequent particles. These vertices, together with 
   $x_\tau$ and $C_0$, are the ones forming $C_1$. 
   
   We then repeat the procedure above. After having moved $x_{k-1}$, we move $x_k$ repetitively until it either reaches $V\setminus B_{L}$ (in which case we set $C_k=C_{k-1}$) 
   or it reaches $C_{k-1}$ (in which case we define $C_k$ as the vertices in $C_{k-1}$ plus all
   vertices visited by $x_k$ since the last sleep instruction $x_k$ sees). The procedure fails if $x_k$ visits $C_{k-1}$ before $V\setminus B_L$ and before seeing any sleep instruction, or if 
   at least one of the subsequent particles $x_{k+1},x_{k+2},\ldots$ is located inside $C_k$.
   For each $k\geq 1$, let $E_k$ be the event that the procedure does not fail when we move $x_k$. 
   Our goal is to show that there exists a positive constant $c$ so that, for all $L\geq 1$, we have 
   \begin{equation}
      \PR\lr{\bigcap\nolimits_{k=1}^{n_L}E_k}\geq c >0,
      \label{eq:claim}
   \end{equation}
   where $n_L$ is the number of particles initially inside $B_L$.
   When~\eqref{eq:claim} holds, we have that this procedure stabilizes $B_L$, without using any instruction at the origin. By the zero-one law (Lemma~\ref{prop:lemma4}), this implies that ARW fixates almost surely.
   
   In order to establish~\eqref{eq:claim}, we will relate this stabilization procedure with a branching process $\mathcal{B}$ on $\mathbb{Z}$, 
   and compare $\mathcal{B}$ with the branching process defined in the beginning of the section, which we denote by $\mathcal{B}'$.
   The processes $\mathcal{B}$ and $\mathcal{B}'$ start with the same number of tokens, and at the same locations. 
   Then we couple $\mathcal{B}$ and $\mathcal{B}'$ to show that at any time
   we can associate each token $\iota$ of $\mathcal{B}$ with a distinct token $\iota'$ of $\mathcal{B'}$ such that 
   \begin{equation}
      \text{the position of $\iota$ is not smaller than that of $\iota'$.}
      \label{eq:brprop}
   \end{equation}
   If at any time the stabilization procedure fails, then we halt $\mathcal{B}$. 
   We will show that this can only happens when $\mathcal{B}'$ has a token at some position $k\leq 0$. 
   This will allow us to use Lemma~\ref{lem:branch} to establish~\eqref{eq:claim}, as the construction of $\mathcal{B}$ and its coupling with 
   $\mathcal{B}'$ will imply that 
   \begin{align}
      \PR\lr{\bigcap\nolimits_{k=1}^{n_L}E_k}
      \geq \PR\lr{\text{no token of $\mathcal{B}'$ visits position $k\leq 0$}}.
      \label{eq:b}
   \end{align}
   
   The initial configuration of $\mathcal{B}$ will be 
   $d$ tokens at position $1$. Each token is associated with one connected component of the graph obtained from $B_L$ by removing $C_0$;
   we denote this graph by $G\setminus C_0$. 
   Since $G$ is a tree, a particle that starts in one component of $G\setminus C_0$ cannot jump to a vertex in another component without visiting $C_0$.
   This will imply that tokens evolve independently of one another.
   The initial configuration of $\mathcal{B}'$ will be identical to that of $\mathcal{B}$, 
   and we associate each token of $\mathcal{B}$ to a distinct token of $\mathcal{B}'$. Note that, for this initial 
   configuration~\eqref{eq:brprop} holds.
   
   Let $\mathcal{L}=\{v_1,v_2,\ldots\}$ be an ordered list of the vertices of $B_L\setminus C_0$, where the vertices are sorted according to the order they are 
   visited in a breadth-first search in $G$ starting from the origin. 
   Thus, for any $i<j$, $v_i$ is not furthest away from the origin than 
   $v_j$.
   Given any subset of vertices $S$ of $G$, and any vertex $v$ of $G$ that is not in $S$, we denote by $d_G(v,S)$ as the distance between $v$ and $S$; that is,
   $$
      d_G(v,S) = \min\{\text{distance between $v$ and $u$ in $G$} \colon u\in S\}.
   $$
   One important point is that we will not sample yet the locations of the particles $x_1,x_2,\ldots$. The procedure for updating $\mathcal{B}$ is to 
   consider the vertex $v$ that is at the front of the $\mathcal{L}$, and check whether $v$ hosts a particle. 
   If the vertex does host a particle, then we move that particle according to the stabilization procedure, 
   which will cause the token corresponding to the component of $v$ to move and/or branch. Depending on the outcome, we may remove $v$ from the list. 
   Then we iterate this procedure. 
   
   We will now describe precisely how $\mathcal{B}$ is updated, and later will show how to couple $\mathcal{B}'$ with $\mathcal{B}$. 
   Assume that, at some moment, the vertex that is at the front of the list is $v$, 
   and that we have already discovered and moved the particles $x_1,x_2,\ldots,x_{j}$.
   Assume that the number of tokens of $\mathcal{B}$ is the number of connected components of $G\setminus C_j$, that 
   $\mathcal{B}'$ has at least as many tokens as $\mathcal{B}$, and that property~\eqref{eq:brprop} holds. 
   Let $S$ be the connected component of $v$ in $G\setminus C_j$, 
   and let $\iota$ be the token of $\mathcal{B}$ corresponding to $S$.
   Assume that $\iota$ is located at position $d_G(v,C_j)$, and that $\iota'$ is the token of $\mathcal{B}'$ associated with $\iota$. 
   Note that all the properties above hold for the initial configuration.

   We start the update of $\mathcal{B}$ by checking whether $v\in C_j$. If this is the case, the stabilization procedure failed, 
   so we halt $\mathcal{B}$.
   Otherwise, we update $\mathcal{B}$ in two phases. 
   Recall that the initial number of particles at $v$ is distributed according 
   to $\mathrm{Geo}^*_{1-\mu}$. We will observe the particles at $v$ one-by-one. This means that in the first time we consider $v$ we will check whether 
   $v$ hosts at least one particle, which happens with probability $\mu$. If this is the case, we move this particle and update $\mathcal{B}$ 
   according to the two phases described below. 
   Then, in the next update of $\mathcal{B}$, we will consider $v$ again, and ask whether $v$ hosts at least two particles. The conditional probability of this 
   event given that $v$ hosts at least 
   one particle is again $\mu$. If this happens, we then move this second particle as we will describe below. 
   We iterate this, each time checking whether $v$ hosts another particle, until we find out 
   that $v$ does not host another particle. At this point we move to another vertex.
   We now describe the two phases for the update of $\mathcal{B}$. The first phase 
   starts by checking whether $v$ hosts another particle.  
   
   First phase, case one: no other particle at $v$. As explained above, this happens with probability $1-\mu$, regardless of how many times $v$ 
   has been checked before. In this case, we remove $v$ from $\mathcal{L}$,
   do not change $\mathcal{B}$, and jump to the second phase.
   
   First phase, case two: there is a new particle $x_{j+1}$ at $v$. This happens with probability $\mu$. 
   We move $x_{j+1}$ according to the stabilization procedure (see Figure~\ref{fig:branch}).
   If $x_{j+1}$ does not visit $C_j$ before leaving $B_L$, nothing happens, and we jump to the second phase.
   Otherwise, let $M_{j+1}$ be the number of vertices that $x_{j+1}$ visits from the last sleep instructions it sees until visiting $C_j$. 
   If the stabilization procedure fails\footnote{Note that one of the reasons for the stabilization procedure to fail 
   is because the set of newly corrupted sites contains a particle 
   that has not yet been moved. We cannot detect that the procedure fails because of this at this moment, because the locations of 
   the particles that have not yet been moved are not known. But this can be detected in future updates of $\mathcal{B}$, when checking whether $v\in C_j$.
   If this is the case, we only halt $\mathcal{B}$ at that time.}, 
   then we halt $\mathcal{B}$. Otherwise, we obtain a set $C_{j+1}$ which satisfies $C_{j+1}\setminus C_j \subset S$.
   We update $\mathcal{B}$ by removing the token $\iota$ and replacing it by as many tokens as the number of connected components of 
   $S\setminus C_{j+1}$. For each such component $S'$ of $S\setminus C_{j+1}$, let the token corresponding to that component be located at position 
   $\min_{u\in\mathcal{L}\cap S'} d_G(u,C_{j+1})$. 
   Since each instruction is a sleep instruction with probability $\frac{\lambda}{1+\lambda}$, we have that 
   \begin{equation}
      \text{$M_{j+1}$ is stochastically dominated by $\mathrm{Geo}_{\frac{\lambda}{1+\lambda}}$}.
      \label{eq:mell}
   \end{equation}
   Also, note that 
   \begin{equation}
      \text{$C_{j+1}$ splits $S$ into at most $d\,M_{j+1}$ connected components}.
      \label{eq:mell2}
   \end{equation}
   The reason is that for each new connected component $S'$, there must exist an edge in $G$ 
   from $S'$ to $C_{j+1}\setminus C_j$, which is a set with $M_{j+1}$ vertices. Also, each new token created in $\mathcal{B}$ has position 
   at least $d_G(v,C_j)-M_{j+1}$, where we recall that $d_G(v,C_j)$ was the position of $\iota$ in $\mathcal{B}$.
   Then we go to the second phase without removing $v$ from the list. 
   \begin{figure}[tbp]
      \begin{center}
         \hspace{\stretch{1}}
         \includegraphics[scale=1.2]{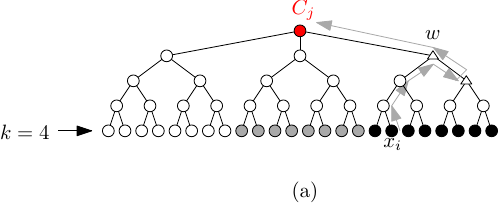}
         \hspace{\stretch{1}}
         \medskip
         
         \includegraphics[scale=1.2]{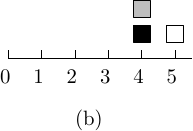}
         \hspace{\stretch{1}}
         \includegraphics[scale=1.2]{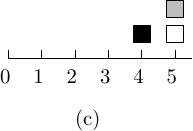}
         \hspace{\stretch{1}}
         \includegraphics[scale=1.2]{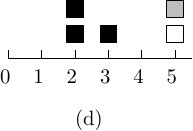}
      \end{center}\vspace{-.5cm}
      \caption{Illustration of how to update $\mathcal{B}$ when observing vertices at distance $k=4$ from $C_j$ (red vertex). After all white vertices in~(a) have been observed, 
      $\mathcal{B}$ is as illustrated in~(b), with a white/grey/black square corresponding to the token of the component with white/grey/black vertices at distance $k$ from $C_j$.
      Then, after observing all grey vertices, which led to no particle visiting $C_j$, the grey token in $\mathcal{B}$ advances one position, as in~(c).
      Next, when observing the black vertices we find particle $x_{j+1}$, which moves according to the grey arrows and sees a sleep instruction in 
         its first visit to $w$. Then we add the triangular vertices to the set of corrupted vertices, 
         and the black token branches into three tokens, one for each connected component created by removing the triangular vertices 
         from the component of the black vertices. After this, $\mathcal{B}$ becomes the configuration in~(d).}
      \label{fig:branch}
   \end{figure}
   
   Second phase. 
   Let $C$ be the set of corrupted sites after the end of the first phase; 
   that is, $C=C_j$ if there were no other particle at $v$ (first case above), or $C=C_{j+1}$ (second case above).
   We check whether $\mathcal{L}$ contains any vertex $u$ in the same component of the graph $G\setminus C$ as $v$, 
   and whose distance (in $G$) to the origin of $G$ is the same as $v$.
   If this is the case (which includes the case that $v$ remains in $\mathcal{L}$ after the end of the first phase), then nothing is done and we end the 
   second phase.
   Otherwise, and this happens only if there was no additional particle at $v$, we take the token $\iota$ and advance it by one position, moving it to position
   $d_G(v,C)+1$. Refer to Figure~\ref{fig:branch}(a--c). This concludes the second phase.
   
   We then iterate the two phases above until $\mathcal{L}$ becomes empty. Since the vertex at the front of $\mathcal{L}$ is removed at each iteration 
   with probability $1-\mu$, each vertex is processed a finite number of times, almost surely, and the procedure ends almost surely.
   
   Now we couple $\mathcal{B}'$ with $\mathcal{B}$ and show that property~\eqref{eq:brprop} continues to hold.
   Recall that $\mathcal{B}'$ is defined in terms of three parameters $\alpha,\beta$ and $d$, 
   where $d$ is the degree of vertices in $G$. For the other parameters, we set 
   $$
      \alpha = \exp\left(-\frac{\mu}{(1-\mu)(d-1)}\right)
      \quad\text{and}\quad
      \beta = \frac{\lambda}{1+\lambda}.
   $$
   First we estimate the probability that a token $\iota$ is advanced from position $k$ to $k+1$. Let $C$ be the 
   set of corrupted sites at the time the token is advanced, and let $u_1,u_2,\ldots,u_{\kappa}$ be the vertices 
   from the component of $G\setminus C$ that is associated to $\iota$ and satisfy $d_G(u_j,C)=k$ for all $j$; note that $\kappa=(d-1)^{k-1}$.
   Note that $\iota$ only advances by one position if each $u_1,u_2,\ldots,u_{\kappa}$ has no particle that reaches $C$ before leaving $B_L$. 
   If there is a particle at some $u_j$, then the probability that this particle visits $C$ before leaving $B_L$ is at most $p_k$, where $p_k$ is 
   defined in Lemma~\ref{lem:pl}.
   Therefore,
   the probability that $\iota$ advances to $k+1$ is at least 
   \begin{align*}
      \left(\sum\nolimits_{i\geq 0} \mu^i (1-\mu) (1-p_k)^i\right)^{(d-1)^{k-1}}
      &= \left(\frac{1-\mu}{1-\mu+\mu p_k}\right)^{(d-1)^{k-1}}\\
      &\geq \exp\left(-\frac{\mu p_k}{1-\mu} \cdot(d-1)^{k-1}\right)\\
      &= \alpha,
   \end{align*}
   where in the inequality we used that $\frac{1}{1+\epsilon}\geq \exp(-\epsilon)$ for all $\epsilon>0$. Therefore, we can couple $\iota$ and 
   $\iota'$ such that if $\iota'$ advances by one position, so does $\iota$. 
   
   Now if $\iota'$ branches, then $\iota$ may branch or advance. If $\iota$ advances, then property~\eqref{eq:brprop} continues to hold regardless of 
   how $\iota'$ branches, since new tokens of $\mathcal{B}'$ will all have positions smaller than that of $\iota$. 
   If $\iota$ also branches, then by~\eqref{eq:mell} and the value of $\beta$, the number of new tokens created by the branch of $\iota'$ 
   stochastically dominates $d\,M_{j+1}$, which is not smaller than the number of new tokens created by the branch of $\iota$. 
   The positions of the new tokens of $\mathcal{B}'$ are not larger than $d_G(v,C_j)-M_{j+1}$, which is a lower bound for the position of the new tokens 
   of $\mathcal{B}'$. Therefore, we can couple the branch of $\iota$ and the branch of $\iota'$ such that 
   each new token of $\mathcal{B}$ can be associated with one distinct new token of $\mathcal{B}'$, and 
   property~\eqref{eq:brprop} continues to hold. We will update $\mathcal{B}'$ even if the stabilization procedure for the branch of $\iota$ failed because the 
   particle encountered no sleep instruction. Note 
   that $M_{j+1}$ is still well defined in this case, and we can perform the coupling between the branch of $\iota'$ and $M_{j+1}$ as described above.
     
   To conclude the proof, we show that if the stabilization procedure fails, then $\mathcal{B}'$ has a token at some position $k\leq 0$. 
   Assume that the particles $x_1,x_2,\ldots,x_j$ have already been observed, and $C_j$ is the current set of corrupted sites. 
   We perform the two phases above repetitively until we find the first particle $x_i$, $i>j$, which visits $C_j$ before exiting $B_L$.
   Note that $C_j=C_{i-1}$.
   Let $v$ be the vertex where $x_i$ started from, and let $k=d_G(v,C_j)$. Note that $k$ is the position of the token $\iota$ of $\mathcal{B}$ corresponding to the 
   component of $v$ in $G\setminus C_j$, and the corresponding token $\iota'$ of $\mathcal{B}'$ is at position at most $k$ by~\eqref{eq:brprop}.
   The procedure cannot fail before finding $x_i$, and there are two events that can make the procedure fail during the move of $x_i$. 
   The first is if $x_{i}$ visits $C_j$ before seeing any sleep instruction. 
   If this happens, then $M_i\geq k$, and all tokens produced from the branching of $\iota'$ will have position not larger than 
   $d_G(v,C_j)-M_{j+1} \leq 0$.
   The second event is if $x_i$ sees sleep instructions, but the new set of corrupted vertices $C_i$ 
   hosts at least one particle that has not yet been moved. 
   This implies that $C_i$ contains a vertex $u$ that is still in $\mathcal{L}$. 
   But by the ordering of the vertices in $\mathcal{L}$, this implies that $d_G(u,C_j)\geq d_G(v,C_j)$ which gives that $M_i\geq k$.
   Therefore the token in $\mathcal{B}$ corresponding to the component of $u$ will be placed at position at most $0$, implying that $\mathcal{B}'$ will
   have a token at position at most $0$.
   
   It is important to remark that the coupling between $\mathcal{B}$ and $\mathcal{B}'$ described above suggests that the tokens of $\mathcal{B}'$ are 
   not updated in the same order as described in the beginning of the section, where all tokens of $\mathcal{B}'$ were moving or branching 
   once in each discrete step. Indeed, a token 
   from $\mathcal{B}$ could branch many times before another token from $\mathcal{B}$ is moved or branches, because 
   for example a vertex $v$ could host more than one particle that visits the set of corrupted sites before leaving $B_L$. However, one can still apply 
   Lemma~\ref{lem:branch} by reordering the updates of $\mathcal{B}$ and $\mathcal{B}'$ accordingly. Letting $A_k$ denote the vertices of $G$ 
   at distance $k$ from the origin, this is possible since tokens move or branch independently of one 
   another, and because for any given $k$ it will
   take only a finite number of steps to remove from $\mathcal{L}$ all vertices of $A_k$. 
   For any $k$, let $\mathcal{I}'_{k}$ be the set of tokens of $\mathcal{B}'$ at the first time that $\mathcal{L}$ 
   contains no vertex from $A_{k}$. Thus, for any $k$, at the moment that $\mathcal{L}$ has no vertex from $A_k$ 
   we obtain that all the tokens from $\mathcal{I}_{k-1}'$ 
   will have moved or branched at least once, giving that each token of $\mathcal{B}'$ will move or branch in a finite time.
   To conclude the proof, note that $\beta$ depends only on $\lambda$. Hence, given any $d$ and $\beta$, 
   we can find $\mu$ small enough to make $\alpha$ close enough to $1$. 
   Then Lemma~\ref{lem:branch} gives that 
   $\PR\lr{\text{no token of $\mathcal{B}'$ visits position $k\leq 0$}}$ is bounded away from $0$, which with~\eqref{eq:b} concludes the proof.
\end{proof}

\section*{Acknowledgments}
We are thankful to G\'abor Pete, Artem Sapozhnikov and Laurent Tournier
for useful discussions. 
Lorenzo Taggi is grateful to Volker Betz for giving him the possibility to work on this project.

\end{document}